\documentclass[]{article}

\usepackage{amsmath}
\usepackage{amsthm}
\usepackage{amssymb}
\usepackage{array}
\usepackage{algorithm,algorithmic}
\usepackage[nohead,margin=1.1in]{geometry}
\usepackage{color}

\newtheorem{proposition}{Proposition}[section]

\newtheorem{example}{Example}[section]
\newtheorem{remark}{Remark}

\usepackage{graphicx}
\usepackage{subfigure}
\usepackage{subfig}
\usepackage[toc,page]{appendix}

\numberwithin{equation}{section}
\newcommand{\dx}{\mathrm{d}x}
\newcommand{\ds}{\mathrm{d}s}

\begin{document}
\numberwithin{equation}{section}
\numberwithin{figure}{section}
\title{Linearized Reconstruction for Diffuse Optical Spectroscopic Imaging}

\author{Habib Ammari\thanks{Department of Mathematics,
ETH Z\"urich, R\"amistrasse 101, CH-8092 Z\"urich, Switzerland (\texttt{habib.ammari@math.ethz.ch}). }
\and Bangti Jin\thanks{Department of Computer Science, University College London,  London, WC1E 6BT,
UK (\texttt{b.jin@ucl.ac.uk}).}\and
Wenlong Zhang\thanks{Department of Mathematics, Southern University of Science and Technology, 1088
Shenzhen, Guangdong, P.R. China (\texttt{zhangwl@sustc.edu.cn}).}}

\date{}
\maketitle

\begin{abstract}
In this paper, we present a novel reconstruction method for diffuse optical spectroscopic imaging
with a commonly used tissue model of optical absorption and scattering. It is based on linearization and group sparsity,
which allows recovering the diffusion coefficient and absorption coefficient simultaneously, provided that their
spectral profiles are incoherent and a sufficient number of wavelengths are judiciously taken for the measurements. We also discuss the reconstruction
for imperfectly known boundary and show that with the multi-wavelength data, the method can reduce the influence of modelling errors and still
recover the absorption coefficient. Extensive numerical experiments are presented to support our analysis.
\end{abstract}

\def\keywords2{\vspace{.5em}{\textbf{  Mathematics Subject Classification
(MSC2000).}~\,\relax}}
\def\endkeywords2{\par}
\keywords2{94A08, 35R30, 92C55.}

\def\keywords{\vspace{.5em}{\textbf{ Keywords.}~\,\relax}}
\def\endkeywords{\par}
\keywords{diffuse optical spectroscopic imaging, reconstruction algorithm, group sparsity, imperfectly known boundary.}

\section{Introduction}

Diffuse optical spectroscopy (DOS) is a noninvasive and quantitative medical imaging modality,
for reconstructing absorption and scattering properties from optical measurements at multiple wavelengths
excited by the near-infrared (NIR) light. NIR light is particularly attractive for
oncological applications because of its deep tissue penetrance and high sensitivity to
haemoglobin concentration and oxygenation state \cite{CerussiTanamaiHsiang:2011}.
DOS imaging effectively exploits the wavelength dependences of tissue optical properties (e.g. absorption, scattering,
anisotropy, reduced scattering and refractive index), and such dependences have been measured and tabulated for
various tissues \cite{Bashkatov11,Jacques:2013}. It was reported that the spectral dependence of tissue
scattering contains much useful information for functional imaging \cite{Cerussi:2002,SevickChange:1991,
SrinivasanPogue:2005}. For example, dual-wavelength spectroscopy has been widely used
to determine the absorption coefficient and hence the concentrations of reduced
hemoglobin and oxygenated hemoglobin in tissue \cite{SevickChange:1991}. Thus, DOS
imaging holds significant potentials for many biomedical applications, e.g., breast oncology
functional brain imaging, stroke monitoring, neonatal hymodynamics
and imaging of breast tumours \cite{Boverman:2005,Cerussi:2002,CerussiTanamaiHsiang:2011,McBridePogue:1999}.

In many biomedical applications, it is realistic to assume that the absorption coefficient is
linked to the concentrations and the spectra of  chromophores through a linear map
(cf. \eqref{eqn:absorption} below), and the spectra of chromophores are known
from experiments. Then the goal in DOS is to recover the individual concentrations from the
measurements taken at multiple wavelengths. This task represents one of the most fundamental
problems arising in accurate functional and molecular imaging.
Theoretically, very little is known about uniqueness and stability of DOS (see
\cite{laure, bookhybrid, AmmariTriki:2017} for related work in electrical impedance tomography (EIT) and
\cite{alberti, jugnon1, jugnon2, BalRen:2012} for quantitative photoacoustic imaging).

Note that despite the linear dependence of absorption on the concentrations, the
dependence of imaging data on the absorption coefficient is highly nonlinear.
Further, it may suffer from severe ill-posedness and possible nonuniqueness; the latter is
inherent to diffuse optical tomography of reconstructing simultaneously the
absorption and diffusion coefficients \cite{Arridge99}. Thus, the imaging problem
is numerically very challenging. Nonetheless, there have been many important efforts
in developing  effective reconstruction algorithms using multi-wavelength data to obtain
images and estimates of spatially varying concentrations of chromophores inside an
optically scattering medium such as biological tissues. These include straightforward
least-squares minimization \cite{Cerussi:2002}, models-based minimization \cite{LauferCoxZhang:2010}
and Bayesian approach \cite{PulkkinenCox:2014}. Generally, there are two different ways
to use the spectral measurements. One is to recover the optical parameters
at each different wavelength separately and then fit the spectral parameter model to
these optical parameters \cite{BalRen:2012,Cox:09, RazanskyDistel:2009}; and the other
is to express the optical parameters as a function of the spectral parameter model and then
estimate the spectral parameter directly \cite{LauferCoxZhang:2010,PulkkinenCox:2014,YuanJiang:2009}.
We refer the interested readers to the survey \cite{Cox12} and the references therein
for detailed discussions.

In this paper, we shall develop a simple and efficient linearized reconstruction method for DOS to
recover the absorption and diffusion coefficients. We employ the diffusion approximation to the
radiative transfer equation for light transport, which has been used widely in biomedical optical
imaging \cite{Arridge99}. Our main contributions are as follows. First, we show that within the
linearized regime, incoherent spectral dependence allows recovering the concentrations and diffusion
coefficient simultaneously, provided that a sufficient number of measurements are judiciously taken.
However, generally there is no explicit criterion on the number of wavelength, except in a few
special cases (see Remark \ref{rmk:decouple}). Second, we demonstrate that with multi-wavelength measurements, the
chromophore concentrations can be still be reasonably recovered even if the domain boundary is only imperfectly
known. Thus, DOS can partially alleviate the deleterious effect of modeling errors, in a manner
similar to multifrequency EIT \cite{AmmariWen2016}. Third and last, these analytical
findings are verified by extensive numerical experiments, where the reconstruction is performed
via a group sparse type recovery technique developed in \cite{AmmariWen2016}.

The rest of the paper is organized as follows. In Section \ref{sec:linmodel}, we derive the linearized
model, and discuss the conditions for simultaneous recovery (and also the one group sparse reconstruction
technique). Then in Section \ref{sec:unknownboundary}, we demonstrate the potential of the multi-wavelength data
for handling modeling errors, especially imperfectly known boundaries. We show that the chromophore concentrations can
still be reasonably recovered from multi-wavelength data, but the diffusion coefficient is lost due to
the corruption of domain deformation. Extensive numerical experiments are carried out in Section
\ref{sec:numer} to support the theoretical analysis. Finally, some concluding remarks are provided in Section \ref{sec:concl}.

\section{The linearized diffuse optical spectroscopy model}\label{sec:linmodel}
In this section, we mathematically formulate the linearized multi-wavelength
method in DOS.
\subsection{Diffuse optical tomography}
First, we introduce the diffuse optical tomography (DOT) model.
Let $\Omega\subset \mathbb{R}^d$ $(d = 2, 3)$ be an open bounded domain with a smooth boundary $\partial\Omega$.
The photon diffusion equation for the photon fluence rate $u$ in the frequency domain takes the following form:
\begin{equation}
\label{eqn:DOT}
\left\{
\begin{aligned}
-\nabla\cdot (D(x,\lambda)\nabla u) + \mu_a(x,\lambda)u &= 0 \quad&\mbox{in } \Omega, \\
D(x,\lambda)\frac{\partial u}{\partial \nu} + \alpha u&= S \quad \quad&\mbox{on } \partial \Omega,
\end{aligned} \right.
\end{equation}
where $\nu$ is the unit outward normal vector to the boundary $\partial\Omega$, the nonnegative functions
$D$ and $\mu_a$ denote the photon diffusion coefficient and absorption coefficient,
respectively. In practice, the source function $S(x)$ is often taken to be a smooth approximation of the 
Dirac function $\delta_y(x)$ located at $y\in \partial \Omega$ \cite{Ripoll05}. The parameter $\alpha$ in 
the boundary condition is formulated as $\alpha =\frac{1-R}{2(1+R)}$ in DOT model, where $R$ is a 
directionally varying refraction parameter \cite{Arridge99}. Throughout, the parameter $\alpha$ is assumed
to be independent of the wavelength $\lambda$. The weak formulation of problem \eqref{eqn:DOT} 
is to find $u\in H^1(\Omega):=\{ v \in L^2(\Omega): \nabla v \in L^2(\Omega)\}$ such that
\begin{equation*}
\int_{\Omega}D(x,\lambda)\nabla u\cdot \nabla v +\mu_a(x,\lambda)uv \dx+ \int_{\partial\Omega}\alpha uv\ds=\int_{\partial \Omega} Sv\ds, \quad \forall v\in H^1(\Omega).
\end{equation*}

In practice, the coefficients $\mu_a$ and $D$ are actually depending on the light wavelength $\lambda$.
The optical properties of the tissue can be expressed using their spectral representations.
Commonly used spectral models for optical properties can be written as \cite{BalRen:2012, LauferCoxZhang:2010, RazanskyDistel:2009,
YuanJiang:2009}:
\begin{align}
  \mu_a(x,\lambda) & = \sum_{k=1}^K \mu_k(x)s_k(\lambda),\label{eqn:absorption}\\
  \mu_s(x,\lambda) & = \mu_{s,\rm ref}(\lambda/\lambda_{\rm ref})^{-b}. \label{eqn:scatter}
\end{align}
That is, the absorption coefficient $\mu_a$ is expressed as a weighted sum of chromophore concentrations $\mu_k$ 
and the corresponding absorption spectra $s_k$ of $K$ known chromophore. The scattering coefficient $\mu_s$ is given, according
to Mie scattering theory, as being proportional to $\mu_{s,\rm ref}$ and (scattering) power $-b$ of a relative
wavelength $\lambda/\lambda_{\rm ref}$ \cite{Bevilacqua:2000,Graaff:1992,SrinivasanPogue:2005}. The coefficient $\mu_{s,
\rm ref}$ is known as the reduced scattering coefficient at a reference wavelength $\lambda_{\rm ref}$, and
it can be spatially dependent. In many types of tissues (e.g., muscle and skin tissues), the wavelength
dependence of $\mu_s$ has been measured and can often be accurately approximated by $\mu_s(x,\lambda)=
a(x)\lambda^{-b}$, where the exponent $b$ is recovered from experiments \cite{Bashkatov11, Jacques:2013}.

The optical diffusion coefficient $D(x,\lambda)$ is given by $D(x,\lambda)=[3(\mu_a+\mu_s)]^{-1}$.
The condition $\mu_s \gg \mu_a$ is usually considered valid in order to ensure the accuracy of the
diffusion approximation to the radiative transfer equation \cite{Cox12,Cox:09,Arridge99}.  Hence,
we assume below that the diffusion coefficient $D(x,\lambda)$ has the form:
\begin{equation}\label{coe:d}
D(x,\lambda)=d(x)s_0(\lambda),
\end{equation}
where the wavelength dependence $s_0(\lambda)$ is known from experiments.

In DOT experiments, the tissue under consideration is illuminated with $M$ sources,
and measurements are taken at detectors. In this work, we assume for the sake of simplicity that the positions $x_n$ of
the sources and detectors are the same, and are distributed over the boundary $\partial\Omega$.
The spectroscopic inverse problem is to recover the spatially-dependent coefficient $d(x)$ and the
concentrations $\mu_k(x)$ of the chromophores given the measured data $u$ (corresponding to the
known sources $S$) on the detectors distributed on the boundary $\partial\Omega$
measured at several wavelengths $\lambda_i$.
It is well-known that the inverse problem of recovering both coefficients $d(x)$ and $\mu_k(x)$ is quite ill-posed
\cite{Cox:09}, since two different pairs of scattering and absorption coefficients
can lead to identical measured data. The multi-wavelength method is a promising approach to
resolve this challenging nonuniqueness issue. It is also reasonable to assume that the
wavelength dependence $s_0(\lambda)$ and the absorption spectra $s_k(\lambda)$ are linearly
independent so as to distinguish the diffusion coefficient and the chromophore concentrations by
effectively using the information contained in the multi-wavelength data.

\subsection{The linearized diffuse optical spectroscopy model}\label{sect:LinDOT}
Now we derive the linearized DOS model, which plays a crucial role in
the reconstruction technique. We discuss the cases of known and unknown
diffusion coefficient separately.

\subsubsection{Unknown diffusion coefficient}
First, we derive the linearized model with both diffusion coefficient $D(x,\lambda)$ and
concentrations $\mu_k(x)$ being unknown. For simplicity, we assume that the coefficient $d(x)$
is a small perturbation of the background, which is taken to be $1$, i.e.,
\begin{equation*}
    d(x) = 1+ \delta d(x),
\end{equation*}
where the unknown perturbation $\delta d(x)$ has a compact support in the domain $\Omega$ and is small (in
suitable $L^p(\Omega)$ norms).

For the inversion, smooth approximations $S_n$ of the Dirac  masses at $\{\delta_{x_n}\}_{n=1}^N$ are applied
and the corresponding fluence rates $u_n$ are measured on the detectors located at all $x_n$ over the
boundary $\partial\Omega$ to gain sufficient information about the diffusion coefficient $D(x,\lambda)$ and absorption
coefficient $\mu_a(x,\lambda)$. That is, let $\{u_n\equiv
u_n(x,\lambda)\}_{n=1}^N\subset H^1(\Omega)$ be the corresponding solutions to \eqref{eqn:DOT}, i.e.,
\begin{equation}\label{eqn:us}
\int_{\Omega}D(x,\lambda)\nabla u_n\cdot \nabla v +\mu_a( x,\lambda)u_nv \dx+ \int_{\partial\Omega}\alpha u_nv
\ds=\int_{\partial \Omega} S_nv \ds, \quad \forall v\in H^1(\Omega).
\end{equation}

Next we derive the linearized multi-wavelength model for the DOT problem based on an integral representation. Let
$v_m\equiv v_m(\lambda)\in H^1(\Omega)$ be the background solution corresponding to $D(x,\lambda)\equiv s_0(\lambda)$ and
$\mu_a \equiv 0$ with the excitation $S_m$, i.e., $v_m$ fulfils
\begin{equation}\label{eqn:vs}
\int_{\Omega}s_0(\lambda)\nabla v_m\cdot \nabla v\dx+ \int_{\partial\Omega}\alpha v_m v \ds=\int_{\partial \Omega} S_m v \ds, \quad \forall v\in H^1(\Omega).
\end{equation}
Note that unless $s_0(\lambda)$ is independent of the wavelength $\lambda$, the dependence of the background
solution $v_m$ on the wavelength $\lambda$ cannot be factorized out. Taking $v=v_m$ in  \eqref{eqn:us} and
$v=u_n$ in  \eqref{eqn:vs} and subtracting the two identities yield
\begin{equation*}
  s_0(\lambda)\int_\Omega \delta d(x) \nabla u_n \cdot \nabla v_m\dx + \sum_{k=1}^Ks_k(\lambda) \int_\Omega \mu_k(x) u_n  v_m\dx =  \int_{\partial \Omega}( S_nv_m-S_mu_n)\ds.
\end{equation*}
Since $\delta d$ and $\mu_k$s are assumed to be small, we can derive the approximations
$\nabla u_n(x,\lambda)\approx \nabla v_n(x,\lambda)$ and $u_n(x,\lambda)\approx v_n(x,\lambda)$
in the domain $\Omega$ (valid in the linear regime), and hence arrive at the following linearized model
\begin{equation} \label{eqn:integral}
    s_0(\lambda)\int_\Omega \delta d(x) \nabla v_n \cdot\nabla v_m\dx + \sum_{k=1}^Ks_k(\lambda) \int_\Omega \mu_k(x) v_n  v_m\dx =  \int_{\partial \Omega}( S_nv_m-S_mu_n)\ds.
\end{equation}
Note that since
$\int_{\partial \Omega} S_mu_n\ds$ is the measured data on the detector located at $x_m$ and $\int_{\partial
\Omega} S_nv_m\ds$ can be computed given the background spectra $s_0(\lambda)$, the right-hand side of the 
model \eqref{eqn:integral} is completely known and can be readily computed.

The DOT imaging problem for the linearized model is to recover $\delta d$ and the chromophore
concentrations $\{\mu_k\}_{k=1}^K$ from $\{u_n(x,\lambda)\}_{n=1}^N$ on the boundary $\partial\Omega$ at several
wavelengths $\{\lambda_q\}_{q=1}^Q$. For the reconstruction, we divide the domain $\Omega$ into a shape
regular quasi-uniform mesh of elements $\{\Omega_l\}_{l=1}^L$ such that $\overline\Omega=\cup_{l=1}^L\Omega_l$,
and consider a piecewise constant approximation of the coefficient $\delta d(x)$ and the concentrations
$\{\mu_k\}_{k=1}^K$ of the chromophores as follows
\begin{equation*}
\begin{aligned}
\delta d(x) &\approx \sum_{l=1}^L (\delta d)_l\chi_{\Omega_l}(x),\\
\mu_k(x) &\approx \sum_{l=1}^L (\mu_k)_l\chi_{\Omega_l}(x),\quad k=1,\ldots,K,
\end{aligned}
\end{equation*}
where  $\chi_{\Omega_l}$ is the characteristic function of the $l$th element $\Omega_l$, and $(\mu_k)_l$
denotes the value of the concentration $\mu_k$ of the $k$th chromophore in the $l$th element $\Omega_l$, so is $(\delta d)_l$.
Upon substituting the approximation into \eqref{eqn:integral}, we have a
finite-dimensional linear inverse problem
\begin{equation*}
s_0(\lambda)\sum_{l=1}^L(\delta d)_l\int_{\Omega_l}  \nabla v_n\cdot \nabla v_m\dx +  \sum_{k=1}^K s_k(\lambda)\sum_{l=1}^L(\mu_k)_l\int_{\Omega_l}  v_n  v_m\dx =
\int_{\partial \Omega}( S_nv_m-S_mu_n)\ds.
\end{equation*}

Finally, we introduce the sensitivity matrix $M^0(\lambda)$, $M^1(\lambda)$ and the data vector $X$.
We use a single index $j=1,\ldots,J$ with  $J=N^2$ for the index pair $(m,n)$
with $j=N(m-1)+ n$, and introduce the sensitivity matrix $M^0(\lambda)=[M^0_{jl}]
\in\mathbb{R}^{J\times L}$ and $M^1(\lambda)=[M^1_{jl}]
\in\mathbb{R}^{J\times L}$ with its entries $M^1_{jl}$ given by
\begin{equation*}
\begin{aligned}
  M^0_{jl}(\lambda)=\int_{\Omega_l} \nabla v_n\cdot \nabla v_m \dx \quad (j\leftrightarrow (m,n)),\\
    M^1_{jl}(\lambda)=\int_{\Omega_l} v_n v_m \dx \quad (j\leftrightarrow (m,n)),
\end{aligned}
\end{equation*}
which is independent of the wavelength $\lambda$.
Likewise, we introduce a data vector $X(\lambda)\in\mathbb{R}^J$ with its $j$th entry $X_j(\lambda)$ given by
\begin{equation*}
  X_j(\lambda) = \int_{\partial \Omega}( S_nv_m-S_mu_n)\ds \quad (j\leftrightarrow (m,n)).
\end{equation*}
By writing the vectors $A_0 = (\delta d)_l\in\mathbb{R}^L$ and $A_k = (\mu_k)_l\in\mathbb{R}^L$, $k=0,\ldots,K$, we obtain the following
linear system (parameterized by the light wavelength $\lambda$)
\begin{equation}\label{eqn:lin-inverse}
 M^0(\lambda)s_0(\lambda)A_0 + M^1(\lambda) \sum_{k=0}^Ks_k(\lambda) A_k  = X(\lambda).
\end{equation}

\subsubsection{Known diffusion coefficient}
If the diffusion coefficient $D(x,\lambda)$ is known, then the goal  is to recover the concentrations $\{\mu_k\}_{k=1}^K$ of the
chromophores. As before, we assume the unknowns $\mu_k$ are small (in suitable $L^p(\Omega)$ norms).
We can  repeat the above procedure, except that the background solution $v_m\in H^1(\Omega)$ is now defined by
\begin{equation}\label{eqn:vscasknow}
  \int_{\Omega}D(x,\lambda)\nabla v_m \cdot\nabla v\dx+ \int_{\partial\Omega}\alpha v_m v \ds=\int_{\partial \Omega} S_m v \ds, \quad \forall v\in H^1(\Omega).
\end{equation}
Taking $v=v_m$ in  \eqref{eqn:us} and $v=u_n$ in  \eqref{eqn:vscasknow} and
subtracting the two identities gives
\begin{equation*}
 \sum_{k=1}^Ks_k(\lambda) \int_\Omega \mu_k(x) u_n  v_m\dx =  \int_{\partial \Omega}( S_nv_m-S_mu_n)\ds.
\end{equation*}
Using the approximation $u_n(x,\lambda)\approx v_n(x,\lambda)$ in the domain $\Omega$
(which is valid in the linear regime), we arrive at the following linearized model
\begin{equation*}
     \sum_{k=1}^Ks_k(\lambda) \int_\Omega \mu_k(x) v_n  v_m\dx=  \int_{\partial \Omega}( S_nv_m-S_mu_n)\ds.
\end{equation*}

As before, we divide the domain $\Omega$ into a shape regular quasi-uniform mesh of elements
$\{\Omega_l\}_{l=1}^L$ such that $\overline\Omega=\cup_{l=1}^L\Omega_l$, and consider piecewise
constant approximations of the chromophore concentrations $\mu_k$:
\begin{equation*}
\begin{aligned}
\mu_k(x) \approx \sum_{l=1}^L (\mu_k)_l\chi_{\Omega_l}(x),\quad k=1,\ldots,K.
\end{aligned}
\end{equation*}
Then we obtain the following finite-dimensional linear inverse problem
\begin{equation*}
 \sum_{k=1}^K s_k(\lambda)\sum_{l=1}^L(\mu_k)_l\int_{\Omega_l}  v_n  v_m\dx = \int_{\partial \Omega}( S_nv_m-S_mu_n)\ds.
\end{equation*}

Using the sensitivity matrix $M$ and the data vector $X$ in \eqref{eqn:lin-inverse}, we get the following parameterized linear system
\begin{equation}\label{eqn:lin-inversescaknow}
M^1(\lambda) \sum_{k=1}^Ks_k(\lambda) A_k  = X(\lambda).
\end{equation}

\subsection{The linearized DOT with multi-wavelength data}
In the two linearized DOT inverse problems in Section \ref{sect:LinDOT}, the vectors $A_0$ (if $d(x)$ is unknown) and
$\{A_k\}^K_{k=1}$ are the quantities
of interest and are to be estimated from the wavelength dependent data $X(\lambda)$, given the
spectra $s_0(\lambda)$ and $s_k(\lambda)$. These quantities directly contain the information of the locations and supports of
$\delta d(x)$ and all the chromophores $\mu_k$. Now we describe a procedure
for recovering the coefficient $d(x)$ (if unknown) and the concentrations $\mu_k(x)$ of the chromophores simultaneously.
The diffusion wavelength dependence $s_0(\lambda)$ is always known (e.g., $s_0(\lambda)=c\lambda^{b}$,
where the parameter $b$ is known from experiments \cite{Bashkatov11}).

We first formulate the inversion method for the case an unknown diffusion coefficient $D(x)$. We consider
the case when all the absorption coefficient spectra $\{s_k(\lambda)\}_{k=1}^K$ are known.
Suppose that we have collated the measured data at $Q$ distinct wavelengths $\{\lambda_q\}_{q=1}^Q$. We write
$S=(s_k(\lambda_q))\in\mathbb{R}^{K\times Q}$, $S_0=(s_0(\lambda))\in\mathbb{R}^{1\times Q}$. We also
introduce the measured matrix $X=[X^t(\lambda_1)\ \ldots \ X^t(\lambda_Q)]^t\in\mathbb{R}^{J\times Q,1}$ and
the vector of unknowns $A=[A^t_1\ \ldots\ A^t_K, A^t_0]^t\in \mathbb{R}^{L\times (K+1),1}$. Here, the superscript
$t$ denotes the matrix/vector transpose. Then we define a total sensitivity matrix $M$ constructed by $Q\times (K+1)$ blocks. The
$(i,j)$th block of $M$ is $M^1(\lambda_i)s_j(\lambda_i)$ for $1\leq i \leq Q$, $1\leq j\leq K$ and the
$(i,K+1)$th block of $M$ is $M^1(\lambda_i)s_0(\lambda_i)$ for $1\leq i \leq Q$. Hence,  \eqref{eqn:lin-inverse} yields
the following linear system:
\begin{equation}\label{eqn:lin-inv-knownprofile}
  MA=X.
\end{equation}

Similarly, when the diffusion coefficient $D(x,\lambda)$ is known, we define another total sensitivity
matrix $M$ constructed by $Q\times K$ blocks. The $(i,j)$th block of $M$ is $M^1(\lambda_i)s_j(\lambda_i)$
for $1\leq i \leq Q$, $1\leq j\leq K$ and  the vector of unknowns $A=[A^t_1\ \ldots\ A^t_K]^t\in \mathbb{R}^{L\times K,1}$.
We get from \eqref{eqn:lin-inversescaknow} the following linear system:
\begin{equation}\label{eqn:lin-inv-knownprofilescaknow}
  MA=X.
\end{equation}

\begin{remark}\label{rmk:decouple}
Under the condition that the wavelength dependence of $M^0(\lambda)$ and $M^1(\lambda)$ can be factorized out
(e.g., $s_0$ is independent of $\lambda$) (and thus can be absorbed into the spectra $s_k(\lambda)$), the linear systems
can be decoupled to gain  further insight. To see this, we consider the case \eqref{eqn:lin-inv-knownprofile}.
Since all the spectra are assumed to be linearly independent, when a sufficient number
of wavelengths $\{\lambda_q\}_{q=1}^Q$ are judiciously taken in the experiment, the corresponding
spectral matrix $\tilde S^t =\left[S_0^t\ \ S^t \right]$ is incoherent in the sense
that $Q\geq K+1$ and $\mathrm{rank}(S)=K+1$ and $\tilde S$ is also well-conditioned. Then the matrix $\tilde S$
has a right inverse $\tilde S^{-1}$. By letting $\tilde Y=X\tilde S^{-1}$, we obtain
\begin{equation*}
  [M^0A_0,M^1A] = \tilde Y.
\end{equation*}
These are $K+1$ decoupled linear systems. By letting
$\tilde Y=[\tilde Y_0\ \ldots \ \tilde Y_K]\in\mathbb{R}^{J\times(K+1)}$, we have $K+1$ independent
(finite-dimensional) linear inverse problems
\begin{equation}\label{eqn:Y}
\begin{aligned}
  M^0A_0&=\tilde Y_0,\\
  M^1A_k&=\tilde Y_k,\quad k = 1,\ldots,K,
  \end{aligned}
\end{equation}
where $A_0$ represents the diffusion coefficient $\delta d(x)$ and $A_k$ (for $1\leq k \leq K$) represents the $k$th
chromophore $\mu_k(x)$. Note that each linear system determines one and only one unknown concentration $A_k$.
Similarly, for the case of a known diffusion coefficient, the matrix $S$ has a right inverse $S^{-1}$, under the given
incoherence assumption. By letting $Y=XS^{-1}$, we obtain
\begin{equation*}
  M^1A = Y.
\end{equation*}
These are $K$ decoupled linear systems. By letting
$Y=[Y_1\ \ldots \ Y_K]\in\mathbb{R}^{J\times(K+1)}$, we have $K$ independent
(finite-dimensional) linear inverse problems
\begin{equation}\label{eqn:Yscaknow}
  M^1A_k=Y_k,\quad k = 1,\ldots,K,
\end{equation}
where $A_k$ (for $1\leq k \leq K$) represents the $k$th chromophore concentrations $\mu_k(x)$.
\end{remark}

Below, we describe a group sparse reconstruction method developed in \cite{AmmariWen2016}
to solve the ill-conditioned linear systems \eqref{eqn:lin-inv-knownprofile} and \eqref{eqn:lin-inv-knownprofilescaknow}.

\subsection{Group sparse reconstruction algorithm}\label{sect:sparsity}
Upon linearization and decoupling steps (see Remark \ref{rmk:decouple}), one arrives at
decoupled linear systems of the form:
\begin{equation}\label{eqn:lin}
  Dx=b,
\end{equation}
where $D = M_0$ or $M\in\mathbb{R}^{J\times L}$ is the sensitivity matrix, $x=A_k\in\mathbb{R}^{L}$ ($0\leq k \leq K$) is the unknown vector,
and $b=Y_k\in\mathbb{R}^J$ ($0\leq k \leq K$) is a known measured data. These linear systems are often under-determined, and severely
ill-conditioned, due to the inherent ill-posed nature of the DOT inverse problem. We adapt a numerical sparse method developed in
\cite{AmmariWen2016} to solve \eqref{eqn:lin}.

The algorithm takes the following two aspects into consideration:
\begin{enumerate}
  \item[(1)] Under the assumption that the unknowns $\delta d$ and $\mu_k$ are small, we may assume that $x$ is sparse.  This suggests to solve the minimization problem
  \begin{equation*}
    \min_{x\in{\Lambda}} \|x\|_1 \quad \mbox{subject to } \|Dx-b\|\leq \epsilon,
  \end{equation*}
 where $\|\cdot\|_1$ denotes the $\ell^1$ norm of a vector. Here, ${\Lambda}$ represents an admissible constraint on the unknowns $x$, since
  they are bounded from below and above, and $\epsilon>0$ is an estimate of the noise level of $b$.

  \item[(2)] In DOT applications, it is also reasonable to assume that each concentration of chromophore $\mu_k$ is clustered, and this
  refers to the concept of group sparsity. The grouping effect is useful to remove the undesirable spikes typically observed
  for the $\ell^1$ penalty alone.
\end{enumerate}

Now we describe the algorithm, i.e., group iterative soft thresholding, listed in  Algorithm \ref{alg:gist}, adapted
from iterative soft thresholding for $\ell^1$ optimization \cite{DaubechiesDefrise:2004}. Here, $N$ is the maximum
number of iterations,  $w_{lk}$ are nonnegative weights controlling the strength of interaction, and $\mathcal{N}_l$
denotes the neighborhood of the $l$th element. We take $w_{lk}=\beta$, for some $\beta>0$ (default: $\beta=0.5$),
and $\mathcal{N}_l$ consists of all elements in the triangulation sharing one edge with the $l$th element. Since
the solution $x$ is expected to be sparse, a natural choice of the initial guess $x^0$ is the zero vector. The
regularization parameter $\gamma$ plays a crucial role in the performance of the reconstruction quality: the larger
the value $\gamma$ is, the sparser the reconstructed solution becomes. There are several possible strategies to
determine its value, e.g., discrepancy principle and balancing principle, or a trial-and-error manner \cite{ItoJin:2015}.

\begin{algorithm}
  \caption{Group iterative soft thresholding.}
  \label{alg:gist}
  \begin{algorithmic}[1]
    \STATE Input $D$, $b$, $W$, $\mathcal{N}$, $\gamma$, $N$ and $x^0$;
    \FOR {$j=1,\ldots,N$}
     \STATE Compute the proxy $g^j$ by
         \begin{equation*}
           g^j=x^j-s^jD^t(Dx^j-b);
         \end{equation*}
     \STATE Compute the generalized proxy $d^j$ by
      \begin{equation*}
           d^j_l = |g^j_l|^2 + \sum_{k\in \mathcal{N}_l}w_{lk}|g^j_k|^2;
      \end{equation*}
     \STATE Compute the normalized proxy $\bar d^j$ by
     \begin{equation*}
        \bar d^j= \max(d^j)^{-1}d^j;
     \end{equation*}
     \STATE Adapt the regularization parameter $\bar\alpha^j$ by
     \begin{equation*}
         \bar{\alpha}_l^j=\gamma/\bar d^j_l,\quad l=1,\ldots,L;
     \end{equation*}
     \STATE Update $x^{j+1}$ by the group thresholding
     \begin{equation*}
       x^{j+1} = P_{\Lambda}(S_{s^j\bar\alpha^j} (g^j));
      \end{equation*}
     \STATE Check the stopping criterion.
    \ENDFOR
  \end{algorithmic}
\end{algorithm}

Below we briefly comment on the main steps of the algorithm and refer to \cite{AmmariWen2016} for details.
\begin{enumerate}
  \item[Step 3] $g^j$ is a gradient
descent update of $x^j$, and $s^j>0$ is the step length, e.g., $s^j=1/\|D\|^2$.
\item[Step 4] This step takes into account the neighboring influence.
\item[Step 5] $\bar d^j$ indicates a grouping effect: the larger  $\bar d^j_l$ is, the more likely
the $l$th element belongs to the group. 
\item[Step 6] This step rescales $\gamma$ to be  inversely proportional to
$\bar d^j_l$.
\item[Step 7] This step performs the projected thresholding with a spatially variable $\bar\alpha^j$.
$P_\Lambda$ denotes the pointwise projection onto the constraint set $\Lambda$ and $S_\lambda$ for $\lambda>0$ is defined by
$S_\lambda(t) = \max(|t|-\lambda,0)\,\mathrm{sign}(t)$.
\end{enumerate}

In our numerical experiments, we apply Algorithm \ref{alg:gist} to the coupled linear systems
\eqref{eqn:lin-inv-knownprofile} and \eqref{eqn:lin-inv-knownprofilescaknow} directly. This
can be achieved by a simple change to Algorithm \ref{alg:gist}. Specifically, at Step 3 of
the algorithm, instead we compute the gradient of the least-squares functional $\tfrac12\|MA-Y\|^2$ by
\begin{equation*}
  g^j=A^j-s^jM^t(Mx^j-Y).
\end{equation*}
The remaining steps of the algorithm are applied to each component $A_i$ independently. Note that
one can easily incorporate separately a regularization parameter $\gamma$ on each component,
which is useful since the diffusion and scattering coefficients are likely to have different
magnitudes.

\section{Imperfectly known boundary}
\label{sec:unknownboundary}

Now, in order
to show the potentials of DOS for handling modelling errors, we consider the case where the boundary  of the domain of interest is not perfectly known. This is
one type of the modelling errors that occurs whenever the positions of the point sources and detectors or the domain
of interest are not perfectly modelled.

We denote the true but unknown physical domain by $\widetilde \Omega$, and the
computational domain by $\Omega$, which approximates $\widetilde\Omega$.
Next, we introduce a forward map $F:\widetilde \Omega\to \Omega$, $\widetilde x\to x$, which
is assumed to be a smooth orientation-preserving map with a sufficiently smooth inverse
map $F^{-1}: \Omega\to\widetilde \Omega$. We denote the Jacobian of the map $F$ by $J_F$,
and the Jacobian of $F$ with respect to the surface integral by $J^S_{F}$.

Suppose now that the function $\widetilde u_n(\widetilde x,\lambda)$ satisfies problem \eqref{eqn:DOT}
in the true domain $\widetilde \Omega$ with the diffusion coefficient $\widetilde D(\widetilde x,\lambda)$,
absorption coefficient $\widetilde\mu_a(\widetilde x,\lambda)$ and source  $\widetilde S(\widetilde x)$, namely
\begin{equation}\label{eqn:DOT-deformed}
  \left\{\begin{aligned}
    -\nabla_{\widetilde x}\cdot(\widetilde D(\widetilde x,\lambda)\nabla_{\widetilde x}\widetilde u_n(\widetilde x,\lambda)) +  \widetilde\mu_a(\widetilde x,\lambda)\widetilde u_n& = 0 \quad \mbox{ in }\widetilde\Omega,\\
\widetilde D(\widetilde x,\lambda)\frac{\partial \widetilde u_n}{\partial \widetilde\nu}(\widetilde x)  + \alpha\widetilde u_n(\widetilde x)  &= \widetilde S_n(\widetilde x) \quad\qquad \mbox{on } \partial \widetilde\Omega.
  \end{aligned}\right.
\end{equation}
Here, $\widetilde S_n$ is a smooth approximation of the Dirac mass at the true position $\widetilde x_n$.
The wavelength-dependent absorption coefficient $\widetilde\mu(\widetilde x,\lambda)$ also has a separable
form related to the true concentrations $\widetilde \mu_k(\widetilde x)$ of the chromophres:
\begin{equation}\label{eqn:absorp-tilde}
\widetilde\mu_a(\widetilde x,\lambda)=\sum_{k=1}^K \widetilde\mu_k(\widetilde x)s_k(\lambda),
\end{equation}
where $\widetilde \mu_k$ are assumed to be small. Furthermore, the diffusion coefficient $D$ takes the linear form:
\begin{equation}\label{eqn:diff-tilde}
\widetilde D(\widetilde x,\lambda) =s_0(\lambda)(1+\widetilde{\delta d}(\widetilde x)).
\end{equation}
The weak formulation of problem \eqref{eqn:DOT-deformed} is given by: find $\widetilde u_n(\cdot,\lambda) \in H^1(\widetilde\Omega)$ such that
\begin{equation}\label{eqn:un-perturbed}
\int_{\widetilde\Omega}\widetilde D( \widetilde x,\lambda)\nabla_{\widetilde x} \widetilde u_n\cdot \nabla_{\widetilde x} \widetilde v +\widetilde\mu_a( \widetilde x,\lambda)\widetilde u_n \widetilde v \mathrm{d}\widetilde{x}+ \int_{\partial\widetilde\Omega}\alpha \widetilde u_n \widetilde v \mathrm{d}\widetilde s=\int_{\widetilde\Omega} \widetilde S_n \widetilde v \mathrm{d}\widetilde{s}, \quad  \forall\widetilde v\in H^1(\widetilde\Omega).
\end{equation}

In the experimental settings, $\widetilde u_n$ is assumed to be measured on the boundary $\partial\widetilde\Omega$. However, because of the incorrect
knowledge of $\partial\widetilde\Omega$, the measured quantity is in fact  $u_n:=\widetilde u_n \circ F^{-1}$ restricted
to the computational boundary $\partial\Omega$.
Below we consider only the case that the domain $\Omega$ is a small variation of the
true physical one $\widetilde\Omega$, so that the linearized
regime is valid. Specifically,  the map $F:\widetilde\Omega\to\Omega$ is given by $F(\widetilde x)=\widetilde x + \epsilon
\widetilde\phi(\widetilde x)$, where $\epsilon$ is a small
scalar and the smooth function $\widetilde\phi(\widetilde x)$ characterizes the domain deformation. Further, let $F^{-1}(x)
=x+\epsilon \phi(x)$ be the inverse map, which is also smooth.

In order to analyze the influence of the domain deformation on the linearized DOT problem, we introduce
the solution $v_m\in H^1(\Omega)$ corresponding to $\bar D(\lambda,x)\equiv s_0(\lambda)$ and
$\mu_a \equiv 0$ with $S_m$ being a smooth approximation of $\delta_{x_m}$, i.e.,  $v_m$ fulfils
\begin{equation}\label{eqn:vs-perturbed}
\int_{\Omega}s_0(\lambda)\nabla v_m \cdot\nabla v\dx+ \int_{\partial\Omega}\alpha  v_m v \ds= \int_{\partial \Omega} S_m v \ds, \qquad  v\in H^1(\Omega).
\end{equation}

Now we can state the corresponding linearized DOT problem with an unknown boundary.
The result indicates that even for an isotropic diffusion coefficient $\widetilde D$ in the true
domain $\widetilde\Omega$, in the computational domain $\Omega$ the equivalent
diffusion coefficient $D$ is generally anisotropic, and there is an additional perturbation
factor on the boundary $\partial\Omega$.

\begin{proposition}\label{prop:linear-perturb}
Let $\mu_a = \widetilde \mu_a \circ F^{-1}$. The linearized inverse problem on the domain $\Omega$ is given by
\begin{equation}\label{eqn:int-domaindeform}
 \begin{split}
    s_0 (\lambda)\left(\int_\Omega (\delta d(x)+\epsilon\Psi)  \nabla v_n\cdot\nabla v_m\dx \right) + \int_{\partial\Omega}\alpha \epsilon \psi v_nv_m\ds  + \sum_{k=1}^Ks_k(\lambda) \int_\Omega \mu_k(x) v_n  v_m\dx\\
  =  \int_{\partial \Omega} (S_n v_m - S_m u_n)  \ds,
       \end{split}
\end{equation}
for some smooth functions $\Psi:\Omega\to\mathbb{R}^{d\times d}$ and $\psi:\partial\Omega\to\mathbb{R}$, which are independent of the wavelength $\lambda$.
\end{proposition}
\begin{proof}
First, we derive the governing equation for the variable $u_n=\widetilde u_n\circ F^{-1}$
in the domain $\Omega$ from \eqref{eqn:un-perturbed}. Let $v=\widetilde v\circ F^{-1} \in H^1(\Omega)$.
By the chain rule, we have $\nabla_{\widetilde x}\widetilde u_n\circ F^{-1}=(J_F^t\circ F^{-1})\nabla_x
u_n$, where the superscript $t$ denotes the matrix transpose. Thus, we deduce that
\begin{equation*}
  \begin{split}
   & \quad\int_{\widetilde\Omega}\widetilde  D(\widetilde x, \lambda)\nabla_{\widetilde x} \widetilde u_n(\widetilde x)\!\cdot\!\nabla_{\widetilde x} \widetilde v(\widetilde x){\rm d}\widetilde x \\
   & =\! \int_\Omega
  (\widetilde D\circ F^{-1})(x) (J_F^t\circ F^{-1}) (x)\nabla u_n(x)\cdot (J_F^t\circ F^{-1})(x) \nabla v(x) |\det J_{F}(x)|^{-1} \dx\\
  & =\!  \int_\Omega
  (J_F\circ F^{-1})(x) (\widetilde D\circ F^{-1})(x) (J_F^t\circ F^{-1}) (x)\nabla u_n(x)\cdot  \nabla v(x) |\det J_{F}(x)|^{-1} \dx \\
  & =  \int_\Omega
  D(x,\lambda) \nabla u_n(x)\cdot  \nabla v(x) \,\dx,
  \end{split}
\end{equation*}
where the transformed diffusion coefficient $D(x,\lambda)$ is given by \cite{Sylvester:1990,KolehmainenLassas:2005,AmmariWen2016}
\begin{equation*}
  D(x,\lambda) = \left(\frac{J_F(\cdot)\widetilde D(\cdot,\lambda)J_F^t(\cdot)}{|\det J_F(\cdot)|}\circ F^{-1}\right)(x).
\end{equation*}
Similarly, we obtain
\begin{align*}
   \int_{\widetilde\Omega}\widetilde  \mu_a(\widetilde x, \lambda) \widetilde u_n(\widetilde x)\widetilde v(\widetilde x){\rm d}\widetilde x & =  \int_\Omega
  \mu_a(x,\lambda)  u_n(x) v(x) |{\det J_{F}(x)}|^{-1} \,\dx,\\
   \int_{\partial\Omega}\alpha \widetilde u_n \widetilde v \mathrm{d}\widetilde s & =   \int_{\partial\Omega}\alpha  u_n  v |{\det J^S_{F}(x)}|^{-1} \ds.
\end{align*}
Here we use the fact that $\widetilde D\equiv 1$ near the boundary in the second equation, since $\delta d$
is compactly supported in the domain. From \eqref{eqn:un-perturbed}, it follows that $u_n$ satisfies
\begin{align}
\int_{\Omega}D(x,\lambda)\nabla u_n\cdot \nabla v &+\mu_a(x,\lambda)u_nv |\det J_{F}(x)|^{-1}\dx\nonumber\\
  &+ \int_{\partial\Omega}\alpha  u_n  v |\det J^S_{F}(x)|^{-1}\ds=  \int_{\partial \Omega} S_n v \ds, \quad \forall v\in H^1(\Omega).\label{eqn:us-deformed}
\end{align}
Then by choosing $v=v_m$ in \eqref{eqn:us-deformed} and $v=u_n$ in \eqref{eqn:vs-perturbed}, we arrive at
\begin{equation}\label{eqn:integralbou}
  \begin{split}
    \int_\Omega (D(x,\lambda)- s_0(\lambda)) \nabla v_n\cdot \nabla v_m \dx &+ \int_\Omega \mu_a(x) v_n  v_m|\det J_{F}(x)|^{-1} \dx \\
   & + \int_{\partial\Omega}\alpha (|\det J^S_{F}(x)|^{-1}-1) u_nv\ds=  \int_{\partial \Omega} (S_n v_m - S_m u_n)  \ds.
      \end{split}
\end{equation}
Note that $J_F= I + \epsilon J_{\widetilde\phi}$, and
$J_{F^{-1}}=I+\epsilon J_\phi = I-\epsilon J_{\widetilde\phi}\circ F^{-1}+O(\epsilon^2)$, since $\epsilon$ is small.
It is known that $|\det J_F|= 1-\epsilon\mathrm{div}\widetilde\phi +O(\epsilon^2)$
\cite[equation (2.10)]{Hettlich:1995}, we similarly derive $|\det J^S_F|= 1+\epsilon\psi +O(\epsilon^2)$.
Then $D(x,\lambda)$ is given by
\begin{equation*}
  \begin{split}
    D(x,\lambda) & = \widetilde D(\cdot,\lambda)(1+\epsilon\mathrm{div}\widetilde\phi(\cdot))^{-1}(I+\epsilon(J_{\widetilde\phi}(\cdot)+J_{\widetilde\phi}^t(\cdot)))\circ F^{-1}(x)+O(\epsilon^2)\\
    &= \widetilde D(\cdot,\lambda)((1-\epsilon\mathrm{div}\widetilde\phi(\cdot))I + \epsilon(J_{\widetilde\phi}(\cdot)+J^t_{\widetilde\phi}(\cdot)))\circ F^{-1}(x) + O(\epsilon^2) \\
    &= \widetilde D(\cdot,\lambda)(1  + \Psi \epsilon)\circ F^{-1}(x) + O(\epsilon^2),
  \end{split}
\end{equation*}
where $\Psi=(J_{\widetilde\phi}+J_{\widetilde\phi}^t- \mathrm{div}\widetilde\phi I)$ is smooth
and independent of $\lambda$.
This and the linear form of $\widetilde\mu_a(\widetilde x,\lambda)$ in \eqref{eqn:diff-tilde} yield
\begin{align*}
  D(x,\lambda)  \approx s_0(\lambda)I +\epsilon s_0(\lambda)\Psi(x)\quad\mbox{and}\quad
\mu_a(x)|{\det J_{F}(x)}|^{-1}=\sum_{k=1}^K \mu_k(x)s_k(\lambda) + o(\epsilon),
\end{align*}
where we have used the assumption that the $\mu_k$ are small.
Upon substituting the above expressions into \eqref{eqn:integralbou} and the  approximations
$\nabla u_n\approx\nabla v_n$, $u_n\approx v_n$ in the domain, we obtain \eqref{eqn:int-domaindeform}.
\end{proof}

By Proposition~\ref{prop:linear-perturb}, in the presence of an imperfectly known boundary with its magnitude
$\epsilon$ being comparable with the concentrations $\{\mu_k\}_{k=1}^K$ and the perturbation $\delta d$, the perturbed
sensitivity system contains significantly modeling errors resulting from the domain deformation. Consequently, a direct inversion of the
linearized model \eqref{eqn:int-domaindeform} is unsuitable. This issue can be resolved using the multi-wavelength
approach as follows. Since \eqref{eqn:int-domaindeform} is completely analogous to \eqref{eqn:integral}, with the only
difference lying in the additional terms in $s_0(\lambda)$ (corresponding to the diffusion coefficient) and the
edge perturbation $\int_{\partial\Omega}\alpha\epsilon \psi v_nv_m\ds $. However, the edge perturbation on the
boundary $\partial\Omega$ can be treated as unknowns corresponding to an additional spectral profile $s_*(\lambda)\equiv 1$.
Thus, one may apply the multi-wavelength approach to recover the quantities of interest.

Specifically, assume that the spectral profiles $s_0, s_1,...,s_K,$ and $s_*$ are incoherent. Then the method in Section \ref{sect:LinDOT}
may be applied straightforwardly, since the right-hand side is known. However, the diffusion perturbation $\delta d$ will 
never be properly reconstructed, due to the pollution of the error term $\epsilon\Psi$ (resulting from the domain perturbation).
The concentrations of chromophores $\mu_k$ corresponding to the wavelength spectrum $s_k$, $k=1,\ldots,K$ may
be reconstructed, since they are affected by the deformation only through the transformation $\mu_k = \widetilde \mu_k
\circ F^{-1}$. That is, the location and shape can be slightly deformed, provided that the deformation magnitude $\epsilon$ 
is small. Only the information of the diffusion coefficient is affected, and cannot be reconstructed. In summary, 
multi-wavelength DOT is very effective to eliminate the modelling errors caused by the boundary uncertainty, at least in the linearized regime.

We have discussed that the influence of an uncertain boundary in the case where both diffusion
and absorption coefficients are unknown. We can also analyze for the case with a known diffusion
coefficient similarly. Specifically, one may assume the
deformed diffusion coefficient on the domain $\bar D(x,\lambda)= \widetilde D(\widetilde x,\lambda )
\circ F^{-1}$ and repeat the procedure of Proposition \ref{prop:linear-perturb}. We just give the
conclusion: when the diffusion coefficient is known, the domain deformation
contributes to a perturbation inside the spectrum $s_0$, and the boundary deformation
pollutes the known diffusion term, and the concentration of chromophores $\mu_k$ corresponding
to the wavelength spectrum $s_k$, $k=1,\ldots,K$ could be reconstructed.

\section{Numerical experiments}\label{sec:numer}
Now we show some numerical results to illustrate our analytical findings. The general setting for the numerical
experiments is as follows. The domain $\Omega$ is taken to be the unit circle $\Omega=\{(x_1,x_2):x_1^2+
x_2^2<1\}$. There are 16 point sources uniformly distributed along the boundary $\partial\Omega$; see Figure
\ref{fig:domain} for a schematic illustration of the domain $\Omega$, the point sources and the detectors.

Furthermore, we assume that the spectral profile $s_0(\lambda)$ for the diffusion coefficient is $s_0(\lambda)=
0.2\lambda^b$, where the parameter $b$ is known from experiments \cite{Bashkatov11}. In all the examples below,
we take $b=1.5$. We will also see that $\mu_s \gg \mu_a$ is fulfilled in all the numerical examples. We take
a directionally varying refraction parameter $R=0.2$, so that $\alpha =\frac{1-R}{2(1+R)}=1/3$. We use a piecewise
linear finite element method on a shape regular quasi-uniform triangulation of the domain $\Omega$. The unknowns
are represented on a coarser finite element mesh using a piecewise constant finite element basis. We measure the
data $u_n(x_m,\lambda) (:= \int_{\partial \Omega} S_m u_n \ds)$ on the detectors located at $x_m$. The noisy data
$u_n^\delta(x_m,\lambda)$ is generated by adding Gaussian noise to the exact data $u^\dag_n(x_m,\lambda)$
corresponding to the true diffusion coefficient $D(x,\lambda)$ and absorption coefficient $\mu_a(x,\lambda)$ by
\begin{equation*}
 u_n^\delta(x_m,\lambda) = u^\dag_n(x_m,\lambda) + \eta\max_{l} |u^\dag_n(x_m,\lambda)-v_n(x_m,\lambda)| \xi_{n,m},
\end{equation*}
where $\eta$ is the noise level, and $\xi_{n,m}$ follows the standard normal distribution.

\begin{figure}
\begin{center}
\includegraphics[width=0.5\columnwidth]{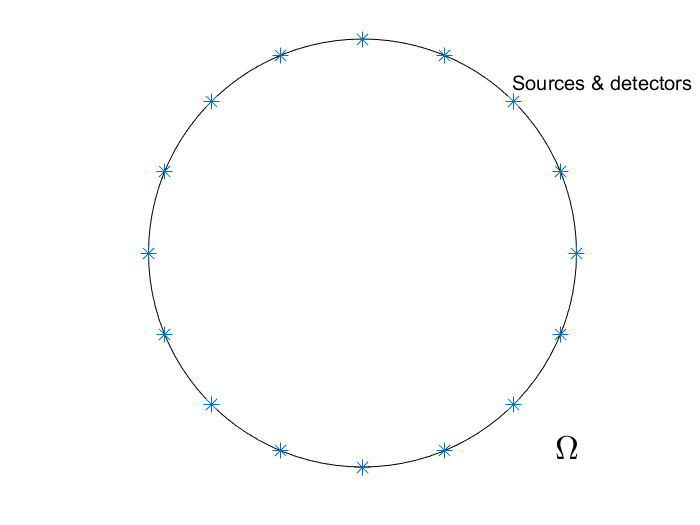}
\end{center}
\caption{The true boundary shape, the positions of sources and detectors.}
\label{fig:domain}
\end{figure}

We present the numerical results for the cases with known boundary and with imperfectly known boundary
separately, and for each case, we  also present the examples with the diffusion coefficient being known
and unknown. In Algorithm~\ref{alg:gist}, we take a constant step size to solve \eqref{eqn:lin}.

\subsection{Perfectly known boundary}

First, we show numerical results for the case with a perfectly known boundary shape. We shall test the robustness
of the algorithm against the noise, and show that the multi-wavelength approach could reduce the deleterious
effects of the noise in the measured data. The regularization parameter $\gamma$ was determined by a trial-and-error
manner, and it was fixed at $\gamma=5\times 10^{-3}$ for the diffusion coefficient and $\gamma=1\times10^{-4}$ for
the absorption coefficient in all the numerical examples with perfectly known boundary. This algorithm is always
initialized with a zero vector.

\begin{example}\label{exam1}
Consider a known diffusion coefficient $D(\lambda,x)=s_0(\lambda)=0.2\lambda^{1.5}$,
and two chromophores inside the domain: the wavelength dependence of the chromophore on the top is $s_1(\lambda)=0.2\lambda$,
and the one on the bottom is $s_2(\lambda)=0.2(\lambda-1)^2$; See Figure \ref{fig:exam1} for an illustration.
We take measurements at $Q=3$ wavelengths with $\lambda_1=1$, $\lambda_2=1.5$ and $\lambda_3=2$.
\end{example}

The numerical results for Example \ref{exam1} are presented in Figure \ref{fig:exam1}. It is observed that the
recovery is very localized within a clean background even with $10\%$ noise in the data, and the supports of the
recovered concentrations of the chromophores agree closely with the true ones and the magnitudes are well-retrieved.
Remarkably, the increase of the noise level from $1\%$ to $10\%$ does not influence much the shape of the recovered
concentrations. Therefore, if the given spectral profiles $s_k(\lambda)$ are sufficiently incoherent, the corresponding
unknowns can be fairly recovered. This example also show that the proposed multi-wavelength approach is very robust
to data noise, due to strong prior imposed by Algorithm \ref{alg:gist}.

\begin{figure}[tbp]
\begin{center}
\begin{tabular}{ccc}
\includegraphics[width=4cm]{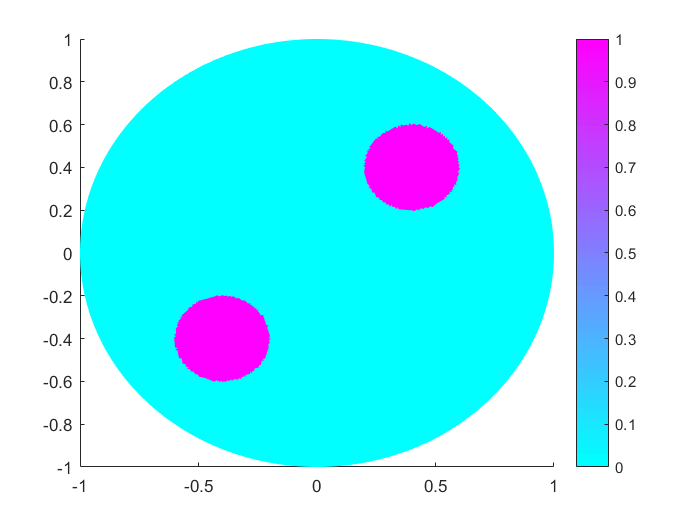} &
\includegraphics[width=4cm]{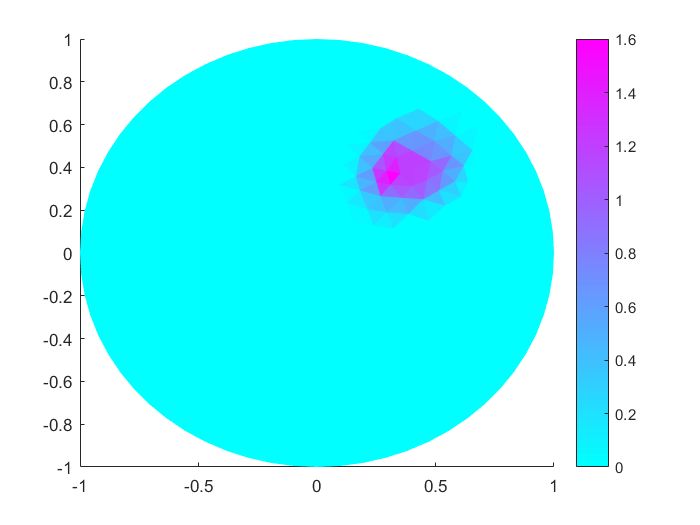} &
\includegraphics[width=4cm]{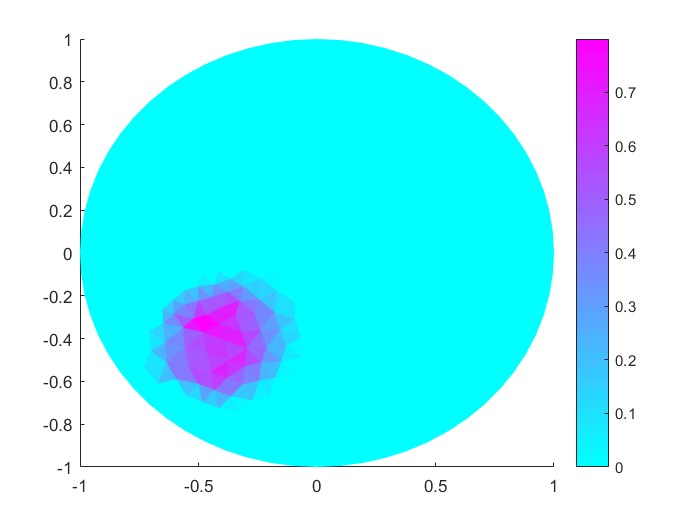}\\
(a) true $\mu_k$  & (b)  recovered $\mu_1$
&(c) recovered $\mu_2$\\
& \includegraphics[width=4cm]{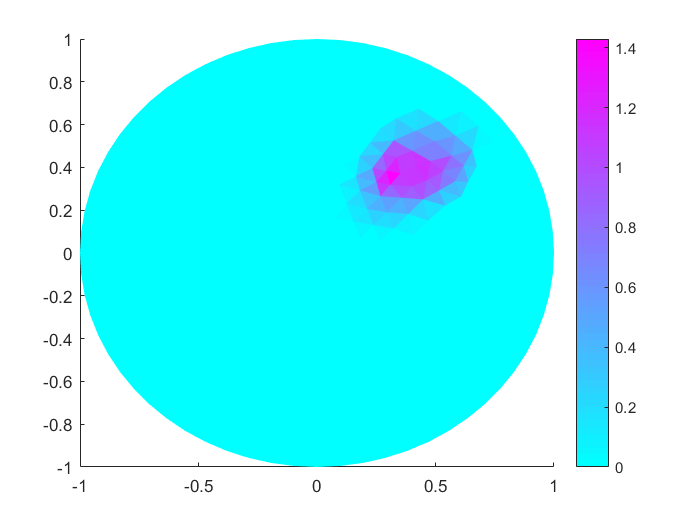} &
\includegraphics[width=4cm]{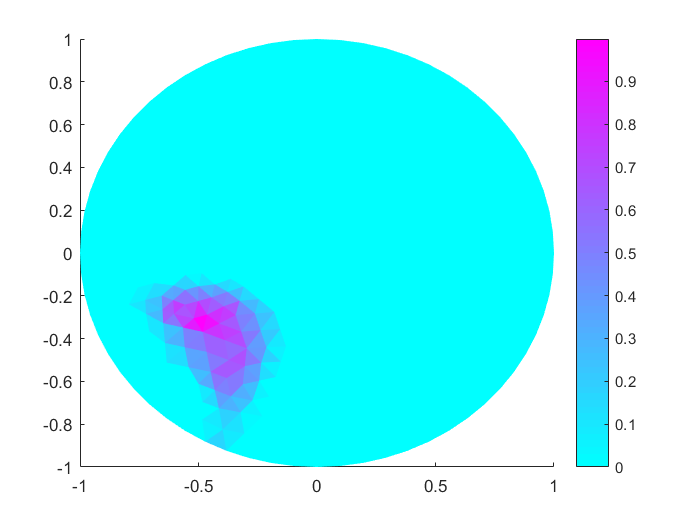} \\
  & (d) recovered $\mu_1$
&(e) recovered $\mu_2$
\end{tabular}
\end{center}
\caption{ Numerical results for Example \ref{exam1}: (a) exact $\mu_1$ and $\mu_2$ of two chromophores;  (b)--(c) recovered results with $\eta=1\%$ noise level; (d)--(e) recovered results with $\eta=10\%$ noise level. }
\label{fig:exam1}
\end{figure}

The next example shows the approach for reconstructing three chromophores inside the domain.

\begin{example}\label{exam5}
Consider the case with a known diffusion coefficient $D(\lambda,x)=s_0(\lambda)=0.2\lambda^{1.5}$,
and 3 chromophores inside the domain:
\begin{itemize}
    \item[(i)] The two chromophores on the top share the wavelength dependence $s_1(\lambda)=0.2\lambda$,
and the one on the bottom has a second spectral profile $s_2(\lambda)=0.2(\lambda-1)^2$;
    \item[(ii)] The wavelength dependence of the chromophore on the top right is $s_1(\lambda)=0.2(\lambda-1)^2$, of the top left one is $s_2(\lambda)=0.2\lambda$
and of the bottom is $s_3(\lambda)=0.2(\lambda-1)^3$.
\end{itemize}
We take measurements at $Q=3$ wavelengths with $\lambda_1=1$, $\lambda_2=1.5$ and $\lambda_3=2$
and the noise level is set to be $\eta=1\%$.
\end{example}

The reconstruction results for Example \ref{exam5} are shown in Figure \ref{fig:exam5}. Figure \ref{fig:exam5}
indicates that the unknowns corresponding to two or three spectral profiles
can be fairly recovered in terms of both the supports and magnitudes. In case (i), the two
chromophores on the top share the wavelength dependence, and they are recovered simultaneously; whereas in
case (ii), the chromophores have three incoherent wavelength dependences, and they
can be recovered separately.

The next example aims at recovering both diffusion and absorption coefficients, which is known to
be very challenging in the absence of multi-wavelength data.

\begin{figure}[hbt!]
\begin{center}
\begin{tabular}{ccc}
\includegraphics[width=4cm]{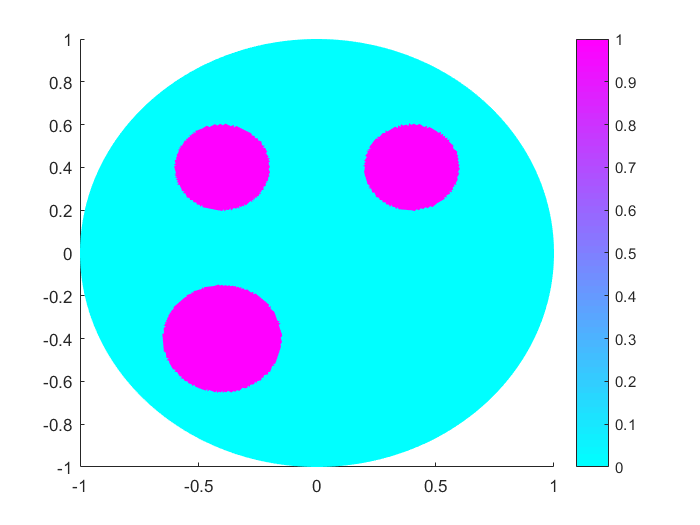} &
\includegraphics[width=4cm]{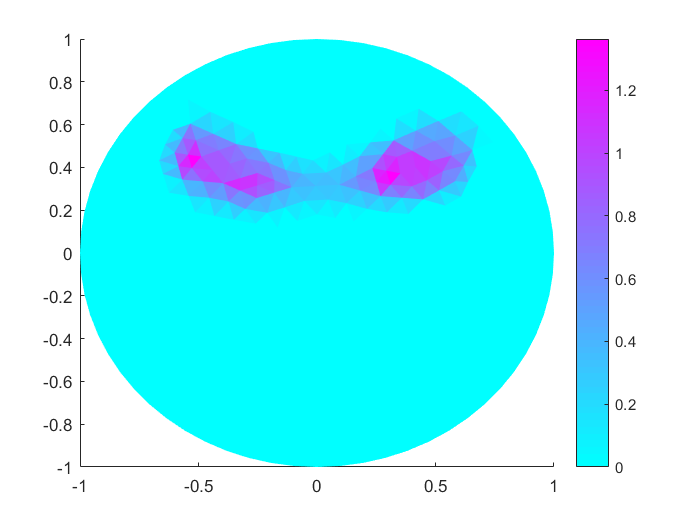} &
\includegraphics[width=4cm]{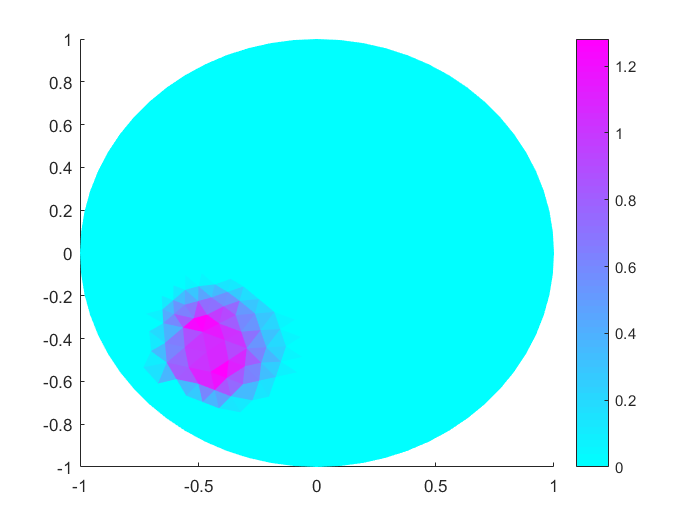}\\
(a) true $\mu_k$  & (b)  recovered $\mu_1$ &(c) recovered $\mu_2$\\
\includegraphics[width=4cm]{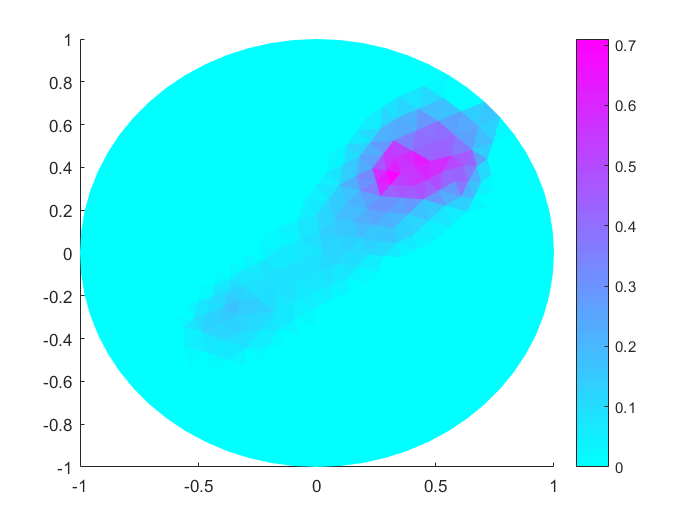} & \includegraphics[width=4cm]{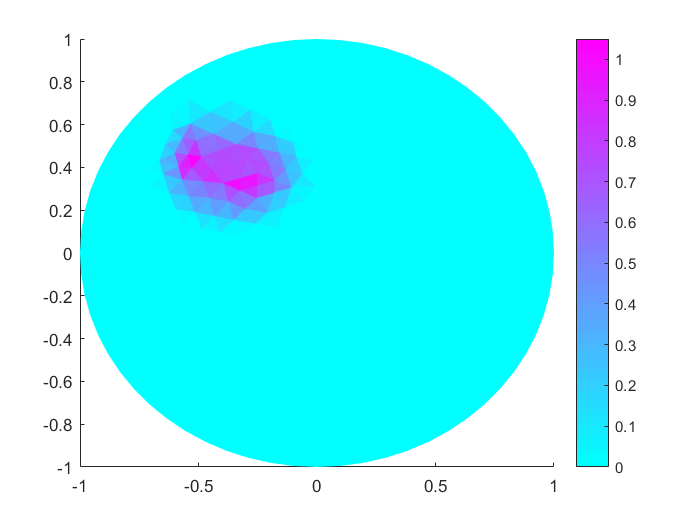} &
\includegraphics[width=4cm]{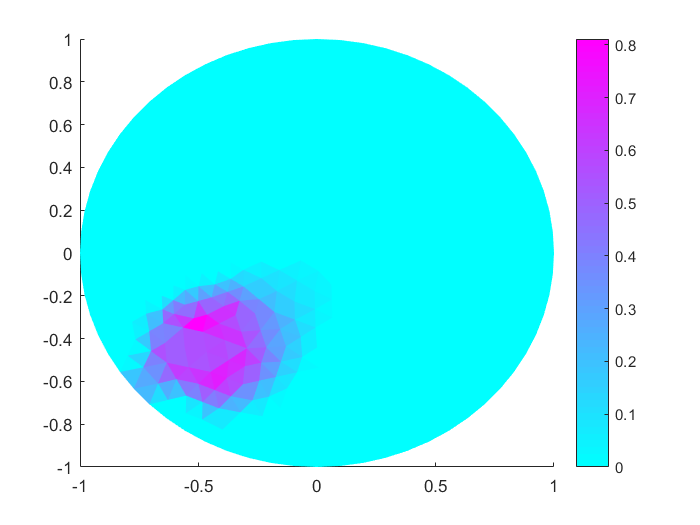} \\
 (d) recovered $\mu_1$ & (e) recovered $\mu_2$
&(f) recovered $\mu_3$
\end{tabular}
\end{center}
\caption{ Numerical results for Example \ref{exam5}: (a) exact $\mu_1$ and $\mu_2$ of two chromophores; (b)-(c) recovered results for case (i) (noise level $\eta=1\%$); (d)-(f) recovered results for case (ii) (noise level $\eta=1\%$). }
\label{fig:exam5}
\end{figure}

\begin{example}\label{exam2}
Consider the case of an unknown diffusion coefficient given by $D(\lambda,x)=s_0(\lambda)(1+0.1\delta d(x))
=0.2\lambda^{1.5} (1+0.25\delta d(x))$. Similar to Example \ref{exam1}, consider two chromophores inside
the domain: the wavelength dependence of the chromophore on the top is $s_1(\lambda)=0.2\lambda$, and that
of the bottom is $s_2(\lambda)=0.2(\lambda-1)^2$. The measurements are taken at $Q=3$ wavelengths with
$\lambda_1=1$, $\lambda_2=1.5$ and $\lambda_3=2$, and the noise level $\eta$ is fixed at $\eta=1\%$.
\end{example}

\begin{figure}[htb!]
\begin{center}
\begin{tabular}{ccc}
\includegraphics[width=4cm]{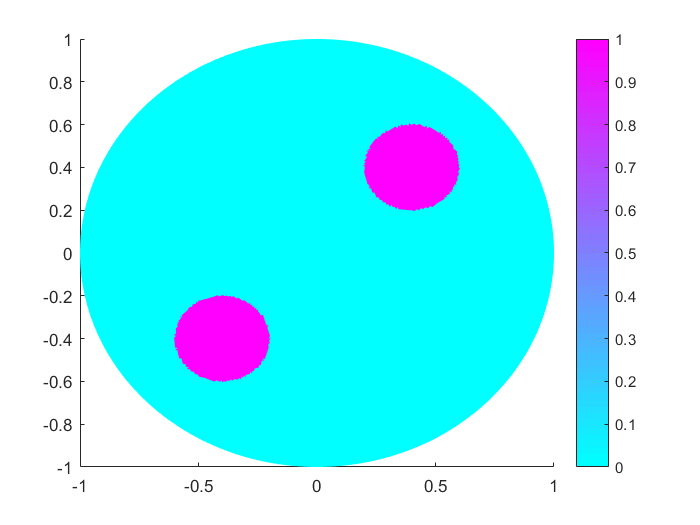} &
\includegraphics[width=4cm]{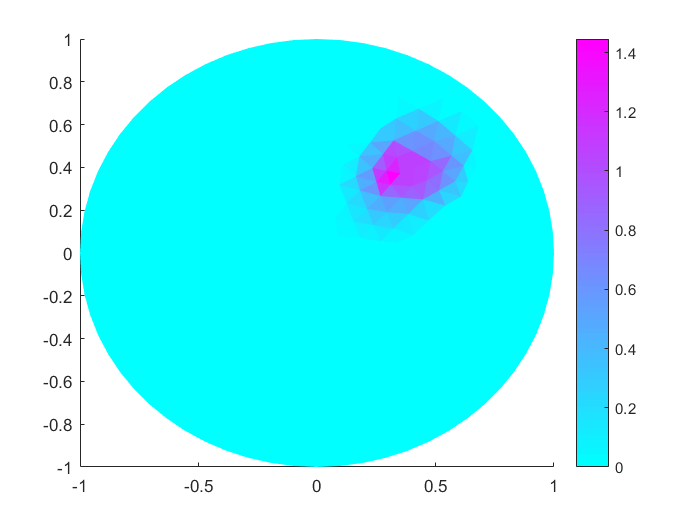} &
\includegraphics[width=4cm]{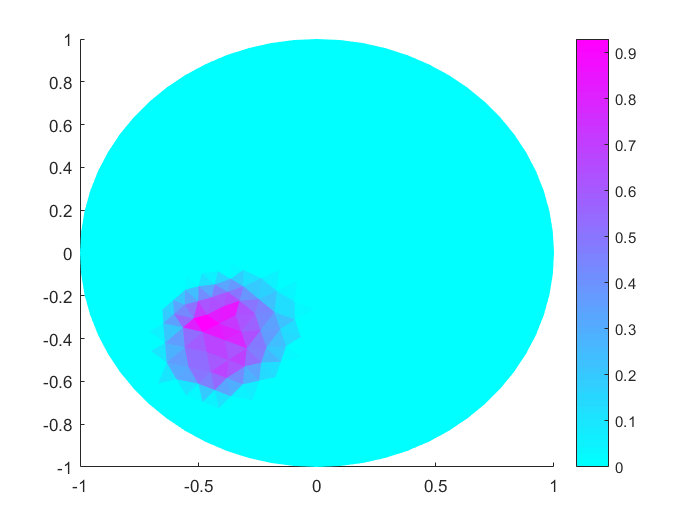}\\
(a) true $\mu_k$  & (b)  recovered $\mu_1$
&(c) recovered $\mu_2$ \\
\includegraphics[width=4cm]{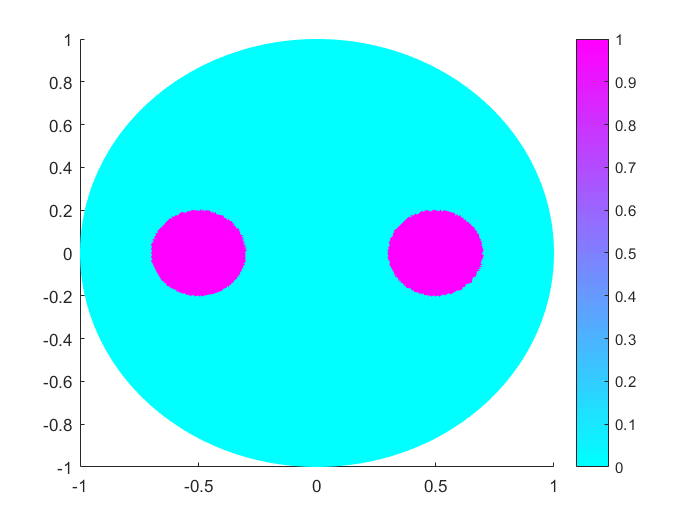} &
\includegraphics[width=4cm]{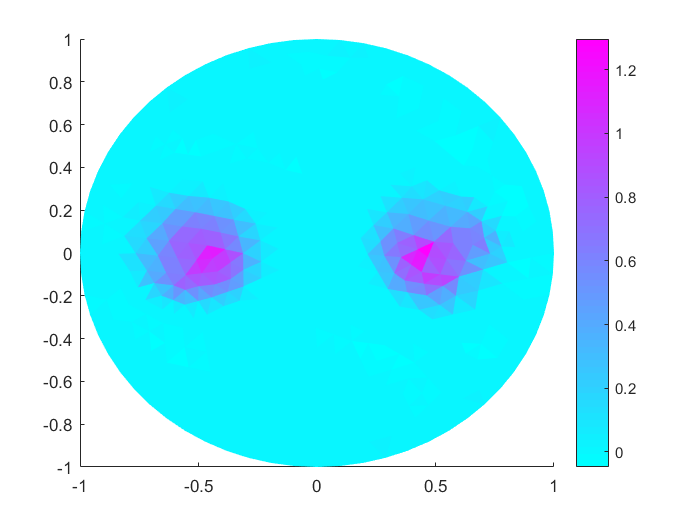} \\
(d) true $\delta d(x)$ & (e) recovered $\delta d(x)$
\end{tabular}
\end{center}
\caption{  Numerical results for Example \ref{exam2}: (a) exact $\mu_1$ and $\mu_2$ of two chromophores;
(b)-(c) recovered results for two chromophores; (d)-(e) true and recovered  $\delta d(x)$. }
\label{fig:exam2}
\end{figure}

The numerical results for Example \ref{exam2} are shown in Figure \ref{fig:exam2}.
Simultaneously reconstructing the diffusion and absorption coefficients is more sensitive to data
noise, when compared with the case of a known diffusion coefficient. Recall that the problem of recovering
both coefficients is quite ill-posed \cite{Cox:09}: two different pairs of scattering and absorption
coefficients can give rise to identical measured data. However, it is observed from Example \ref{exam2}
that the multi-wavelength approach allows overcoming this nonuniqueness issue, provided that the
spectra are indeed incoherent.

The next example shows that multi-wavelength data can mitigate the effects of the noise.
\begin{example}\label{exam6}
Consider the setting of Example \ref{exam2}, but with a noise level $\eta=30\%$. We study two different numbers of wavelengths.
\begin{itemize}
    \item[(i)] The measurements are taken at $Q=3$ wavelengths with $\lambda_i=1+(i-1)/2$, $i=1,\ldots,3$;
    \item[(ii)] The measurements are taken at $Q=30$ wavelengths with $\lambda_i=1+(i-1)/29$, $i=1,\ldots,30$.
\end{itemize}
\end{example}

Numerical results for Example \ref{exam6} are shown in Figure \ref{fig:exam6}.
When using only data for 3 wavelengths, the recovered images are blurred by $ 30\%$ noise.
However, using data for 30 wavelengths, both the diffusion coefficient $\delta d$ and two
chromophore concentrations $\mu_k$ are much better resolved than using data with 3 wavelengths. Hence,
more wavelength observations can greatly mitigate the effects of data noise;
which concurs with the observations from the experimental study \cite{SrinivasanPogue:2005}.

\begin{figure}[hbt!]
\begin{center}
\begin{tabular}{ccc}
\includegraphics[width=4cm]{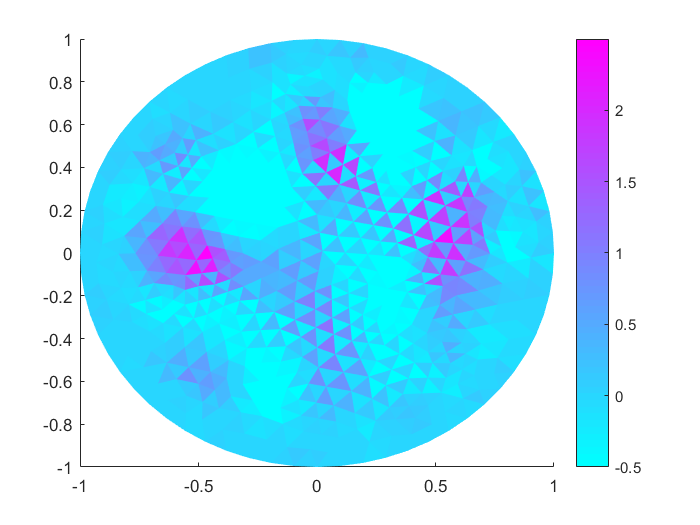} &
\includegraphics[width=4cm]{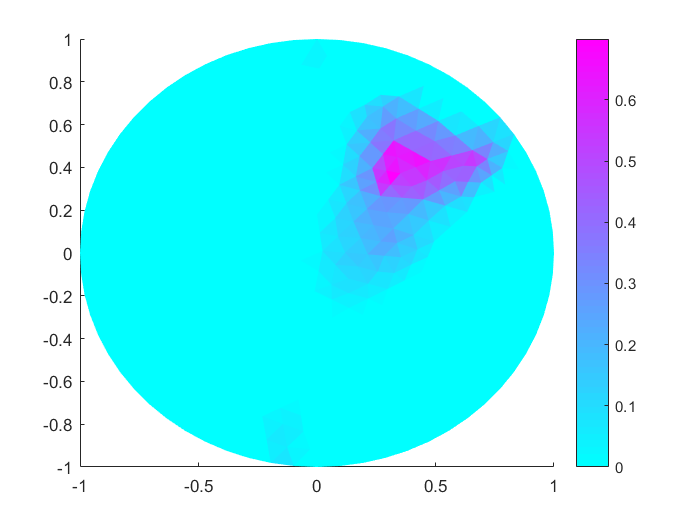} &
\includegraphics[width=4cm]{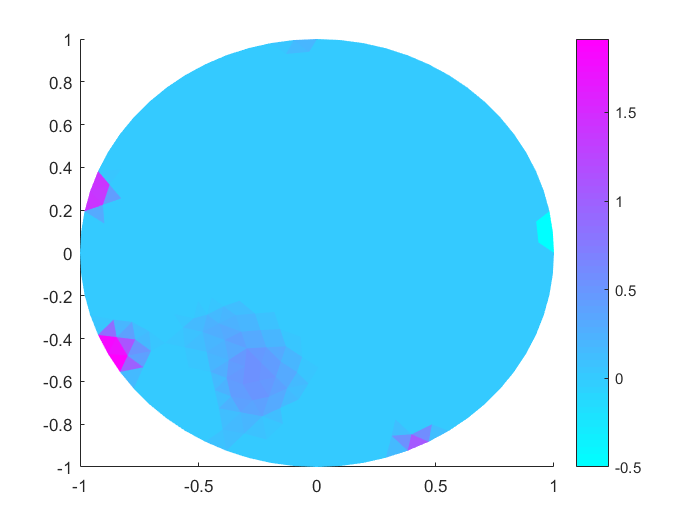}\\
(a) recovered $\delta d(x)$  & (b)  recovered $\mu_1$ & (c) recovered $\mu_2$\\
\includegraphics[width=4cm]{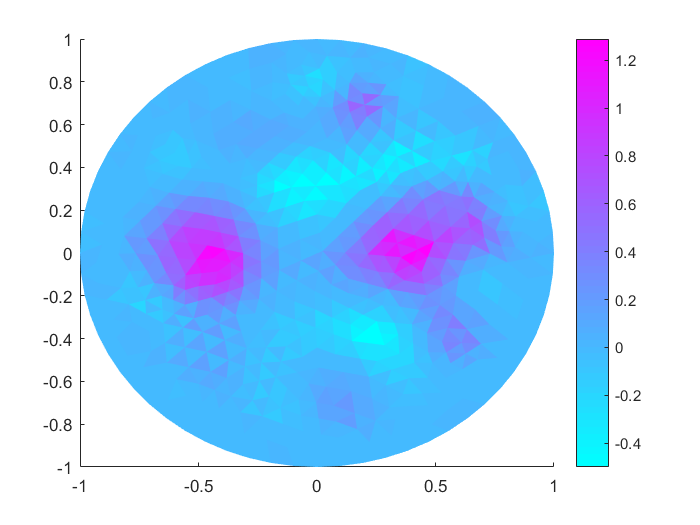} & \includegraphics[width=4cm]{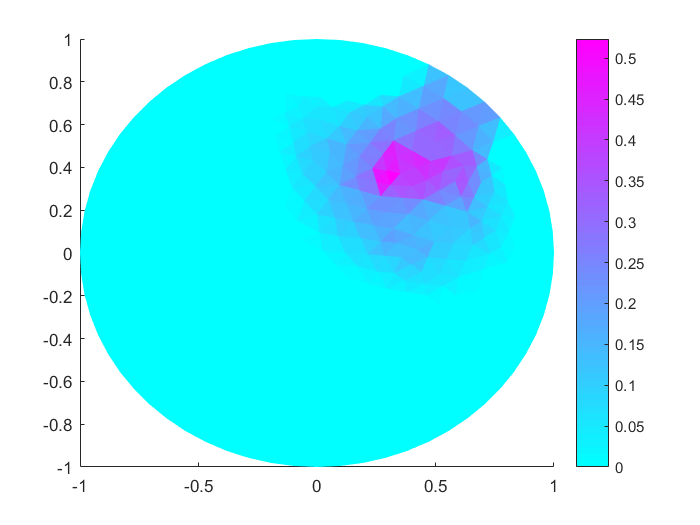} &
\includegraphics[width=4cm]{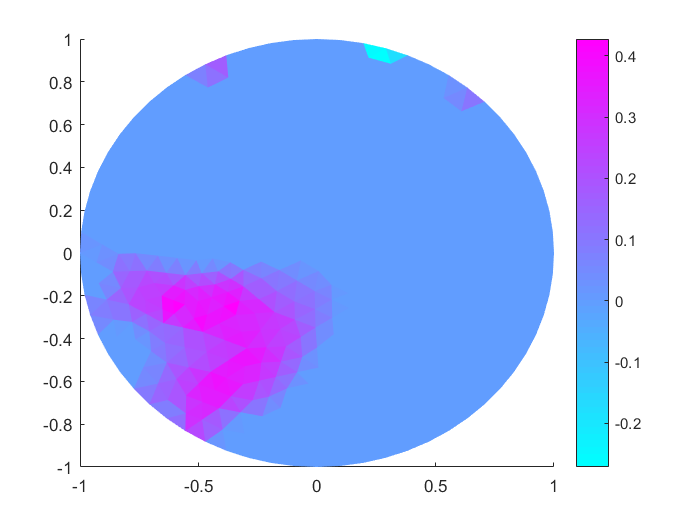} \\
 (d) recovered $\delta d(x)$ & (e) recovered $\mu_1$
&(f) recovered $\mu_2$
\end{tabular}
\end{center}
\caption{ Numerical results for Example \ref{exam6}: (a)-(c) the recovered results for case (i) (with 3 wavelengths); (d)-(f): results for case (ii) (with 30 wavelengths). }
\label{fig:exam6}
\end{figure}

\subsection{Imperfectly known boundary}

Now we illustrate the approach in the case of an imperfectly known boundary. The (unknown) true domain $\widetilde\Omega$
is an ellipse centered at the origin with semi-axes $a$ and $b$, $\mathcal{E}_{a,b} =\{(x_1,x_2): x_1^2/a^2+x_2^2/b^2<1\}$,
and the computational domain $\Omega$ is the unit disk. In this part, the regularization parameter $\gamma$ was determined
by a trial-and-error manner, and it was fixed at $\gamma=5\times 1^{-2}$ for the diffusion coefficient, $\gamma=1\times
10^{-4}$ for the absorption coefficient and $\gamma=1\times10^{-4}$ for the edge perturbation in all the numerical examples
with imperfectly known boundary. This algorithm is always initialized with a zero vector.

\begin{example}\label{exam3}
Consider the case of a known diffusion coefficient $D(\lambda,x)=s_0(\lambda)=0.2\lambda^{1.5}$, and two different shape
deformations: $(i)$ $\widetilde{\Omega}$ is an ellipse with $a=1.1$ and $b=0.9$ and $(ii)$ $\widetilde{\Omega}$ is an
ellipse with $a=1.2$ and $b=0.8$. Consider two chromophores inside $\widetilde{\Omega}$: the wavelength dependence of
the chromophore on the top is $s_1(\lambda)=0.5(\lambda-1)$, and that of the bottom is $s_2(\lambda)=0.5(\lambda-1)^2$.
The measurements are taken at $Q=3$ wavelengths with $\lambda_1=1$, $\lambda_2=1.5$ and $\lambda_3=2$, and
the noise level is fixed at $\eta=1\%$.
\end{example}

The numerical results for Example \ref{exam3} are shown in Figure \ref{fig:exam3}. This example illustrates the influence
of the deformation scale on the reconstruction. The numerical results show clearly the potential of the multi-wavelength
approach: Even using the wrong domain for the inversion step, we can still recover the concentrations of the chromophores
(or more precisely the deformed concentrations $\mu_k = \widetilde \mu_k \circ F^{-1}$). The numerical results also show
even we assume known diffusion coefficient in case (i), we should also use all the spectra $s_0$, $s_k$, and $s_*$ to
recover the right concentrations of the chromophores (Figures \ref{fig:exam3} (d) and (e)), or the results will be ruined
by the shape deformation (Figures \ref{fig:exam3} (b) and (c)).

\begin{figure}[hbt!]
\begin{center}
\begin{tabular}{ccc}
\includegraphics[width=4cm]{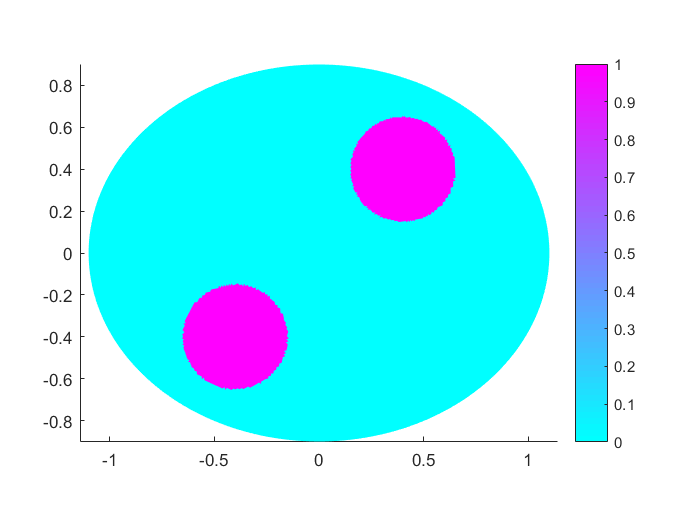} &
\includegraphics[width=4cm]{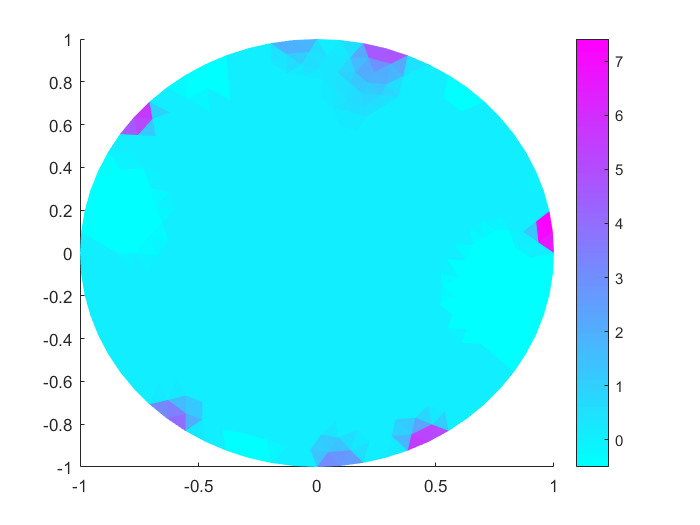} &
\includegraphics[width=4cm]{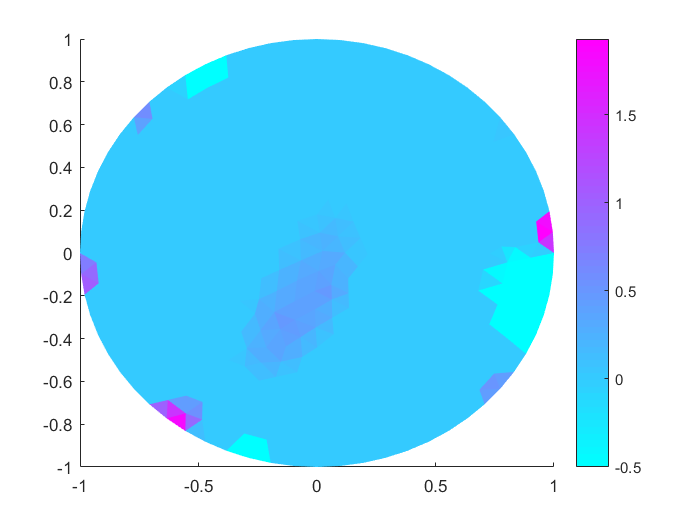}\\
(a) true $\mu_k$  & (b)  recovered $\mu_1$ &(c) recovered $\mu_2$ \\
&\includegraphics[width=4cm]{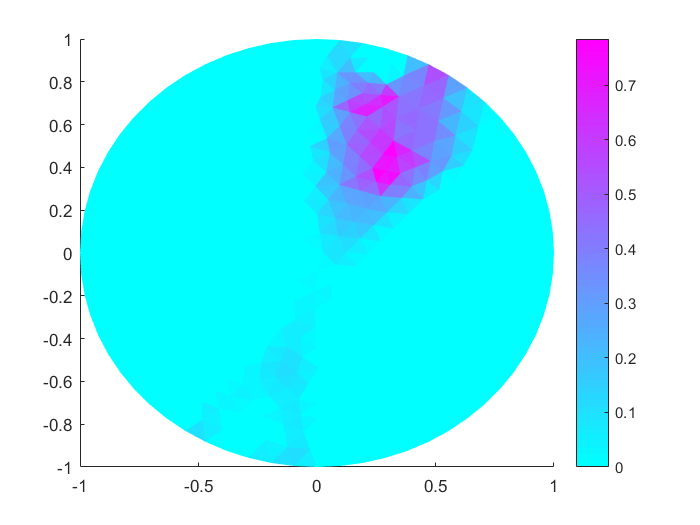} &
\includegraphics[width=4cm]{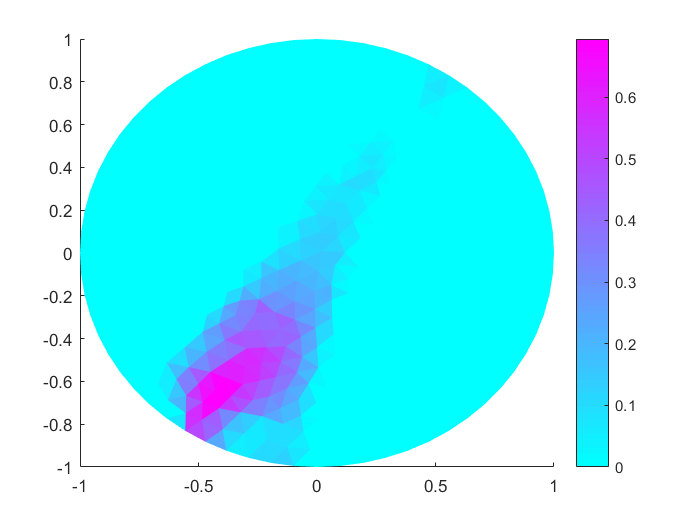}\\
& (d)  recovered $\mu_1$ &(e) recovered $\mu_2$ \\
\includegraphics[width=4cm]{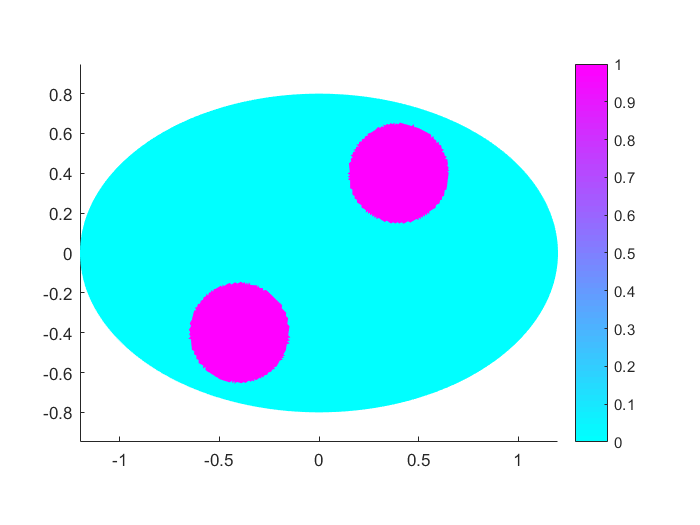} &
\includegraphics[width=4cm]{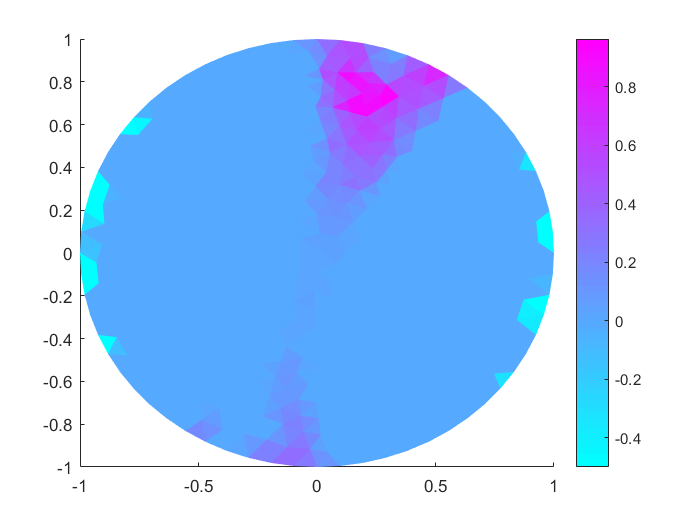} &
\includegraphics[width=4cm]{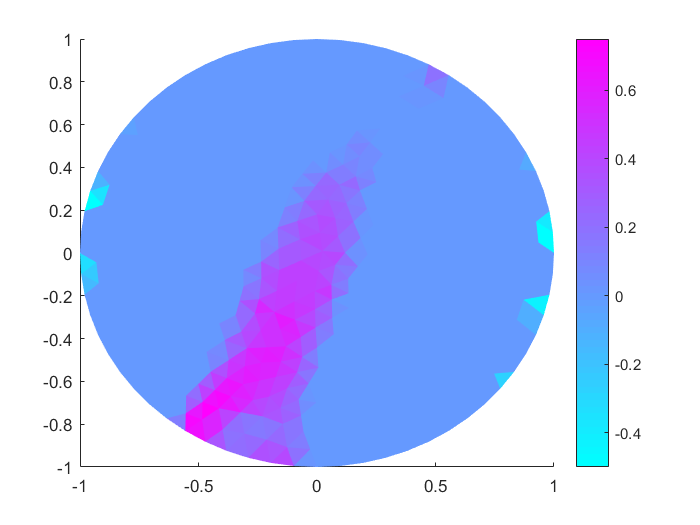}\\
(f) true $\mu_k$  & (g)  recovered $\mu_1$ &(h) recovered $\mu_2$
\end{tabular}
\end{center}
\caption{Numerical results for Example \ref{exam3}: (a),(f) exact $\mu_1$ and $\mu_2$ of two chromophores in $\widetilde\Omega$;
(b)-(c) recovered results for two chromophores in $\Omega$ for case (i) only using the spectra $s_k$ of the chromophores; (d)-(e)
recovered results for two chromophores in $\Omega$ for case (i) using the spectra $s_k$ of the chromophores, the spectrum $s_0$ of
the diffusion coefficient and the spectrum $s_*(\lambda)\equiv 1$ of the edge perturbation; (g)-(h) the recovered results for case
(ii) using all the spectra $s_0$, $s_k$, and $s_*$.} \label{fig:exam3}
\end{figure}

\begin{example}\label{exam4}
Consider the case of an unknown diffusion coefficient $D(\lambda,x)=s_0(\lambda)(1+0.1\delta d(x))=0.2\lambda^{1.5}
(1+0.25\delta d(x))$, and the unknown true domain $\widetilde\Omega$ is an ellipse with $a=1.1$ and $b=0.9$. Consider
two chromophores inside the domain $\widetilde\Omega$: the wavelength dependence of the chromophore on the top is
$s_1(\lambda)=0.5(\lambda-1)$, and that of the bottom is $s_2(\lambda)=0.5(\lambda-1)^2$. The measurements are taken
at $Q=3$ wavelengths with $\lambda_1=1$, $\lambda_2=1.5$ and $\lambda_3=2$, and the noise level is fixed at $\eta=1\%$.
\end{example}

The numerical results for Example \ref{exam4} are shown in Figure \ref{fig:exam4}.
It is observed that the two chromophores are recovered well in spite of the imperfectly
known boundary, while the diffusion coefficient is totally distorted by domain deformation,
and thus it cannot be accurately recovered. The empirical observations on Examples \ref{exam3} and \ref{exam4} concur
with the theoretical predictions in Section \ref{sec:unknownboundary}: Proposition \ref{prop:linear-perturb}
implies that the domain deformation will be added to the recovered results corresponding to the
spectral profile $s_0(\lambda)$ and the diffusion coefficient cannot be recovered due to the
domain deformation, but the deformed chromophore concentrations $\mu_k = \widetilde \mu_k
\circ F^{-1}$ can still be recovered.

\begin{figure}[hbt!]
\begin{center}
\begin{tabular}{ccc}
\includegraphics[width=4cm]{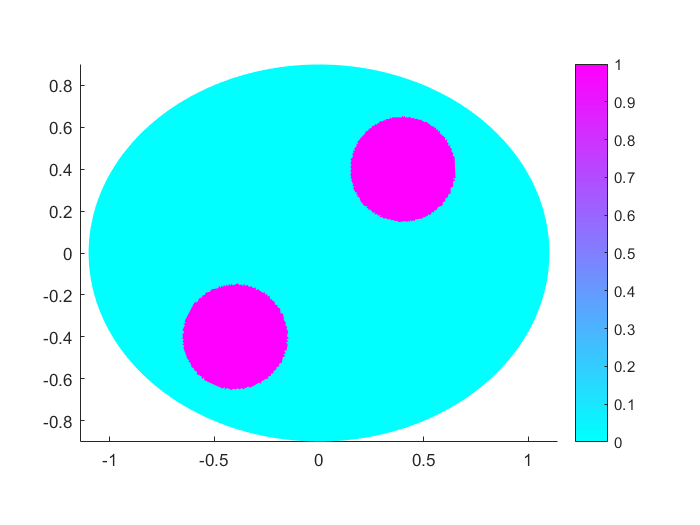} &
\includegraphics[width=4cm]{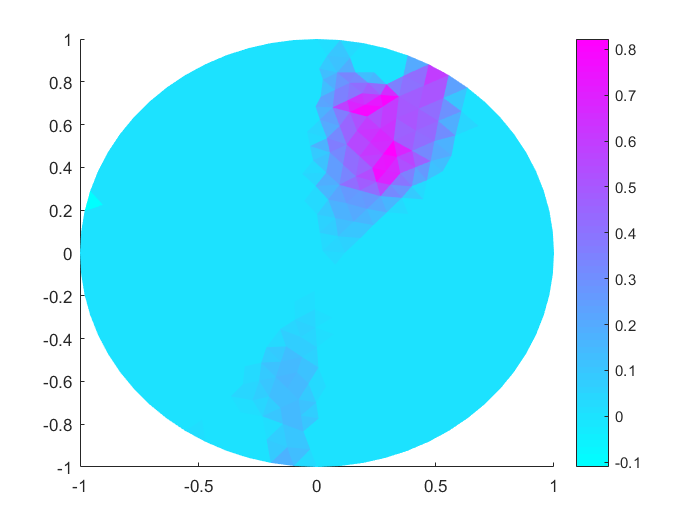} &
\includegraphics[width=4cm]{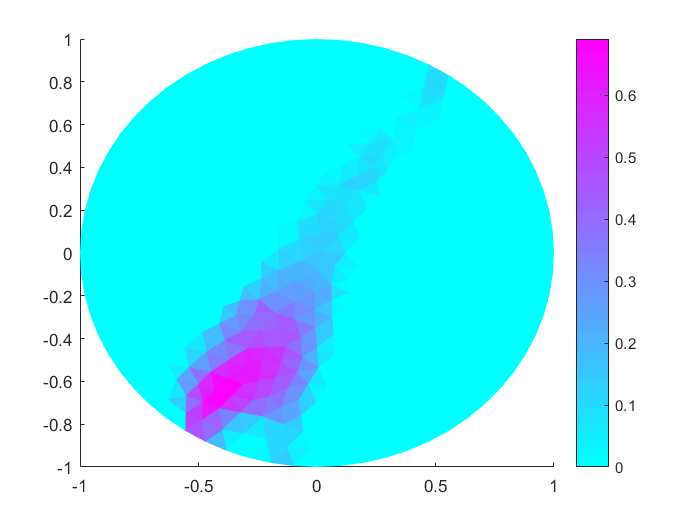}\\
(a) true $\mu_k$  & (b)  recovered $\mu_1$ & (c) recovered $\mu_2$ \\
\includegraphics[width=4cm]{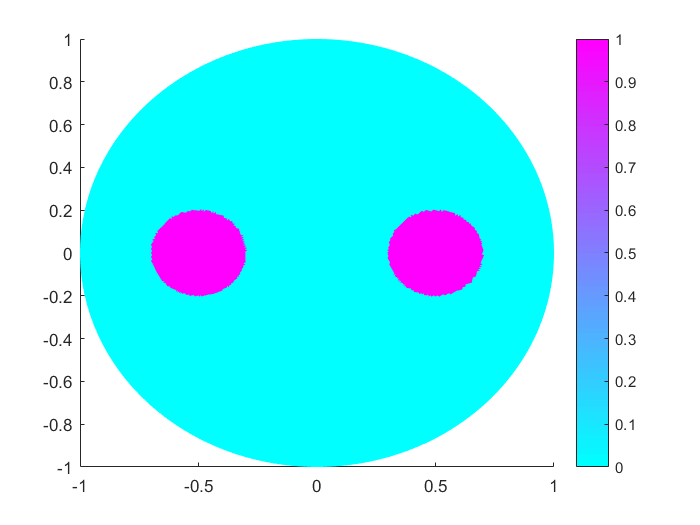} &
\includegraphics[width=4cm]{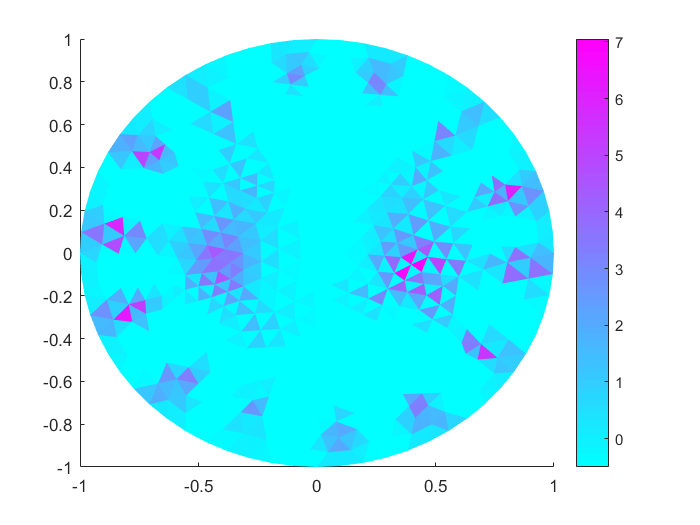} \\
(d) true $\delta d$ & (e) recovered $\delta d$
\end{tabular}
\end{center}
\caption{  Numerical results for Example \ref{exam4}: (a),(d): exact $\mu_1$ and $\mu_2$ of two chromophores in $\widetilde{\Omega}$;
(b)-(c): recovered results for two chromophores in $\Omega$; (e) the recovered result corresponding to $s_0(\lambda)$. }
\label{fig:exam4}
\end{figure}

\section{Conclusion} \label{sec:concl}

In this work, we have introduced a novel reconstruction technique for diffuse optical imaging
with multi-wavelength data. The approach is based on a linearized model and a group sparsity approach.
We have shown that within the linear regime, our reconstruction technique  allows recovering the concentration of individual
chromophore and the diffusion coefficients, provided that their spectral profiles are known and incoherent. Furthermore, we have demonstrated that the multi-wavelength data can significantly reduce
modelling errors associated with an imperfectly known boundary. In fact, it allows recovering
well (deformed) concentrations of the chromophores. These findings are
fully supported by extensive numerical experiments.

\bibliographystyle{abbrv}

\bibliography{multi-wave}

\end{document}